\documentclass[10pt,a4paper,reqno]{amsart}
\usepackage{vmargin,color}
\usepackage[latin1]{inputenc}

\usepackage{amsmath,amsfonts,amsthm,epsfig,graphicx,mathtools,tikz-cd, amssymb,setspace,enumerate,enumitem}
\usepackage[foot]{amsaddr}
\usepackage{caption}
\usepackage{mathrsfs}
\usepackage[colorlinks,linkcolor=blue,citecolor=blue]{hyperref}
\usepackage[capitalise]{cleveref}
\usepackage{autonum}
\usepackage{tikz-cd}
\def\cE{\mathcal{E}}

\def\cY{\mathcal{Y}}
\def\\rho{\mathcal{\rho}}
\def\cP{\mathcal{P}}
\def\P{\cP}

\def\N{\mathbb N}

\def\R{\mathbb R}
\def\T{\mathbb T}

\def\\rho\rho{\mathbf{\rho}}

\def\eps{\varepsilon}

\def\diam{{\rm diam}}

\def\supp{{\rm supp}}

\def\weaker{{\sigma}}

\def\weakstar{\stackrel{*}{\rightharpoonup}}

\DeclareMathAlphabet{\mathup}{OT1}{\familydefault}{m}{n}
\newcommand{\dx}[1]{\mathop{}\!\mathup{d} #1}

\DeclarePairedDelimiter{\abs}{\lvert}{\rvert}
\DeclarePairedDelimiter{\norm}{\lVert}{\rVert}
\DeclarePairedDelimiter{\bra}{(}{)}

\DeclarePairedDelimiter{\set}{\{}{\}}

\newcommand{\Leb}{\ensuremath{{L}}}
\newcommand{\SobH}{\ensuremath{{H}}}

\newtheorem{theorem}{Theorem}

\newtheorem{proposition}[theorem]{Proposition}
\newtheorem{lemma}[theorem]{Lemma}
\newtheorem{corollary}[theorem]{Corollary}

\theoremstyle{definition}

\newtheorem{defn}[theorem]{Definition}
\newtheorem{assumption}{Assumption}

\theoremstyle{remark}
\newtheorem{remark}[theorem]{Remark}
\newenvironment{tenumerate}[1][]
  {\enumerate[label=\textup(\alph*\textup),ref=(\alph*),#1]}
  {\endenumerate}

\newcommand{\Dref}[1]{\dodref\Cref#1\relax}
\def\dodref#1#2!#3\relax{#1{#2}\ref{#2!#3}}

\allowdisplaybreaks

\mathchardef\mhyphen="2D

\numberwithin{equation}{section}
\numberwithin{figure}{section}
\numberwithin{theorem}{section}
\makeatletter
\@namedef{subjclassname@2020}{\textup{2020} Mathematics Subject Classification}
\makeatother

\title[An invariance principle for gradient flows]{An invariance principle for gradient flows \\ in the space of probability measures}
\author{Jos\'e A. Carrillo$^{\dagger}$, Rishabh S. Gvalani$^{\ddagger}$, \& Jeremy Wu$^{\dagger}$}
\date{\today}
\address[$\dagger$]{Mathematical Institute, University of Oxford, Oxford OX2 6GG, UK}
\address[$\ddagger$]{Max-Planck-Institut f\"ur Mathematik in den Naturwissenschaften, 04103 Leipzig, Germany} 

\begin{document}
\maketitle
\begin{abstract}
We seek to establish qualitative convergence results to a general class of evolution PDEs described by gradient flows in optimal transportation distances. These qualitative convergence results come from dynamical systems under the general name of LaSalle Invariance Principle. By combining some of the basic notions of gradient flow theory and dynamical systems, we are able to reproduce this invariance principle in the setting of evolution PDEs under general assumptions. We apply this abstract theory to a non-exhaustive list of examples that recover, simplify, and even extend the results in their respective literatures.
\end{abstract}
\section{Introduction}
Our goal is to show that gradient flows in metric spaces can be cast into the framework of dynamical systems to apply generic qualitative long-time convergence results. In particular, we seek to reformulate the LaSalle invariance principle in the setting of evolution PDEs. Let us provide a formal and instructive review of this concept in finite dimensional ODEs, we refer to~\cite[Chapter 9]{CH98} for a general semigroup description. For $x_0\in \R^d$, consider a curve $x: [0,\infty) \to \R^d$ solving
\[
\frac{\dx{}}{\dx{t}}x(t) = F(x(t)), \quad x(0) = x_0,
\]
where $F :\R^d \to \R^d $ is a sufficiently smooth driving vector field. The \emph{LaSalle invariance principle} in this case reads
\begin{theorem}[LaSalle invariance principle: ODE version]
Suppose $\Phi : \R^d \to \R$ is continuous and satisfies the strict Lyapunov condition; $t\mapsto \Phi(x(t))$ is non-increasing and
\[
\forall t \ge 0, \,  \Phi(x(t)) = \Phi(x_0) \implies \forall t \ge 0, \, x(t) = x_0.
\]
Assume further that the set $ \bigcup_{t\ge 0} \{x(t)\}$ is relatively compact. Then, the distance between $x(t)$ and the set of equilibrium points $F^{-1}(\{0\})$ converges to 0 as $t\to \infty$.
\end{theorem}

If we specialize to \textit{gradient flows}, the situation simplifies significantly. If $E:\R^d\to \R$ is some (smooth) function, let us now consider
\[
\frac{\dx{}}{\dx{t}}x(t) = -\nabla E(x(t)), \quad x(0) = x_0.
\]
In this setting, $E$ is immediately a strict Lyapunov function since
\[
\frac{\dx{}}{\dx{t}}E(x(t)) = -|\nabla E(x(t))|^2 \le 0.
\]
Here, it suffices to check the compactness of trajectories $x(t)$ to apply the invariance principle. The compactness property can also be pivoted onto $E$ since trajectories are contained in sublevel sets of $E$; it suffices to check that sublevel sets of $E$ are relatively compact. We seek to extend these advantages in finite dimensional gradient flows to infinite dimensional gradient flows commonly describing evolution PDEs in biology, physics, social networks, and many other applications. The function $E$ describes quite a lot of the dynamics of the ODEs and this is also true for PDEs which are the infinite dimensional analogues.

Let us formally introduce the basic notion of gradient flows in the space of probability measures by describing a typical example of such an equation. We refer the reader to~\cref{sec:prelim} or~\cite{AGS08} for more details and greater generality. For $\Omega = \R^d$ or $\T^d$, and functions $F : \R_{\ge 0} \to \R; \, V,W : \Omega \to \R$ we consider the functional on probability densities with finite second moment, denoted by $\cP_2(\Omega)$,
\[
E(\rho) = \int_\Omega F(\rho) \dx{x} + \int_\Omega V(x) \dx{\rho}(x) + \frac{1}{2}\int_\Omega (W*\rho)(x) \dx{\rho}(x).
\]
We refer to the following PDE or its solution as Wasserstein gradient flow of $E$
\[
\partial_t\rho = \nabla \cdot \left(
\rho \nabla (F' + V +  W*\rho)
\right).
\]
In fact, following Otto's formal Riemannian calculus~\cite{Ott01}, this PDE can be interpreted closer to the ODE gradient flow by
\[
\frac{\dx{}}{\dx{t}}\rho(t) = -\nabla_{W_2}E(\rho),
\]
where $W_2$ is the 2-Wasserstein distance on $\cP_2(\Omega)$. Here $\nabla_{W_2}$  denotes the gradient operator
\begin{align}
\nabla_{W_2}:= \nabla \cdot \bra*{\rho \nabla \frac{\delta}{\delta \rho}} \, ,
\end{align}
on $(\cP_2(\Omega), W_2(\cdot,\cdot))$, which can be thought of as an infinite-dimensional Riemannian manifold (cf., for example,~\cite{Lot08,Gig12}). By a formal computation, we see that
\[
\frac{\dx{}}{\dx{t}}E(\rho(t)) = - \int_\Omega |\nabla (F' + V + W*\rho)|^2 \dx{\rho}(x) \le 0,
\]
which is in direct analogy with the ODE setting, i.e. $E$ satisfies the strict Lyapunov condition. This property has been the basis of a very active field of research in the last couple of decades. Our focus is to exploit this property at the abstract level of dynamical systems described above. Several difficulties arise before we can formulate a a version of the LaSalle invariance principle in this context. Firstly, the regularity of $E: \rho \mapsto E(\rho)$ is usually limited to lower semicontinuity in the most interesting cases. Secondly, the compactness of trajectories is usually established in a \textit{weaker} topology than the $W_2$ topology. We now list a few prominent examples of Wasserstein gradient flows, some of which we consider in Section~\ref{sec:examples}, are given by
\begin{enumerate}
    \item Fokker-Planck equations (cf.~\cite{JKO98}): $F(x) = x\log x, \, V \in C^\infty(\Omega), \, W = 0$. Assuming $V$ satisfies certain growth assumptions, it is known that solutions converge exponentially in  the 2-Wasserstein distance towards the unique stationary solution given by
    \[
    \rho_s(x) = \frac{e^{-V(x)}}{\int_\Omega e^{-V(y)}\dx{y}}.
    \]
    \item Porous medium (and fast diffusion) equations (cf.~\cite{Ott01,Vaz07}): $$F(x) = \left\{
    \begin{array}{cc}
    \displaystyle x\log x     &m=1  \\
    \dfrac{1}{m-1}x^m     &m > 1 
    \end{array}
    \right.$$ and $V=W=0$. It is known that self-similar solutions exist and give rise to finite or infinite time extinction depending on the value of $m$.
    \item Aggregation equations (cf.~\cite{CDFLS11}): $F = V = 0$ with $W$ typically chosen to be radial. These equations can describe a variety of interaction-based phenomena from biology, physics, and social networks. One of the main topics of interest is what conditions guarantee consensus concentration or segregation.
    \item Parabolic-elliptic Keller-Segel models (cf.~\cite{BCC08,BCL209}): $F = x\log x, V=0$, and $W$ is the negative Green's function in $\R^d$. Understanding how certain infinite-time profiles are generated is a very active area of research. The long-time behaviour of solutions strongly depends on the initial mass.
    \item Machine Learning (c.f.~\cite{W20,WE20} and the references therein): An important step in the implementation of multi-layer deep neural networks is to optimise the network parameters (with respect to some appropriately chosen cost function) by evolving them using a gradient descent scheme. The mean field equation associated to this system corresponds to a gradient flow in the $2$-Wasserstein metric. The long-time asymptotics of the density of the parameters  then describes, in some sense, the optimal choice of parameters in the training.
\end{enumerate}
We postpone further discussion on the literature of some of these particular examples and others to the relevant subsections below (cf.~\cref{sec:examples}).

 This paper seeks to unify the convergence results of many different PDEs under one general theory. However, one of the limitations of this soft approach is the lack of quantitative rates of convergence. We note that deriving explicit rates for convergence usually relies on \'a la carte techniques that are problem-dependent. Furthermore, these techniques only provide convergence of solutions towards the set of equilibrium points. In specific cases, one can extract more information about the structure of this set and thus obtain a stronger convergence results e.g. to a particular stationary solution. In the situation where there are multiple equilibrium points (cf.~\cite{CGPS20,CG19}), we can only hope to obtain local rates of convergence in the possible basins of attraction for stable stationary solutions. Our result gives the first step in this direction; for instance, in the presence of phase transitions, depending on the energy of the initial data, we can identify the limiting stationary solutions (cf.~\cref{sec:kuramoto} and \cref{sec:MV}).

The plan of the paper is as follows: In Section~\ref{sec:prelim}, we develop the notation and assumptions necessary to state the main results of~\cref{thm:omegastat,thm:converge,thm:converge'}. Section~\ref{sec:omegalimit} is dedicated to the proofs of these results. Section~\ref{sec:examples} is dedicated to examples of gradient flows in which our abstract theory applies as well as a small overview of the existing literature in each example. Appendix~\ref{sec:dynsys} contains a brief review of the abstract formalism of dynamical systems in metric spaces with a mild adaptation to allow for non-unique solutions. Appendix~\ref{sec:lsc} contains a technical lemma allowing us to easily check lower semicontinuity of certain functionals, which we repeatedly use in our examples.

\section{Preliminaries and the main abstract result}
\label{sec:prelim}
We denote by $\cP(X)$ the space of all Borel probability measures on some Polish metric space $(X,d)$. We equip this with the so-called narrow or weak topology, i.e. the coarsest topology on $\cP(X)$ such that all functions of the form $\rho \mapsto \int_X \varphi \dx{\rho}$, $\varphi \in C_b(X)$ are continuous, where $C_b(X)$ is the space of bounded, continuous functions on $X$. This topology is metrisable by, for instance, the L\'evy--Prokhorov metric, which we denote by $d_{LP}(\cdot,\cdot)$. Furthermore, we define $\cP_p(X)$ to be the subset of $\cP(X)$ consisting of all probability measures $\rho$ with finite $p$-moment. We equip this space with the topology generated by the so-called $p$-Wasserstein or transportation cost distance, $W_p(\cdot,\cdot)$.  Note that the space $\cP_p(X)$ also inherits a topology from $\cP(X)$ as a subspace of $\cP(X)$. This topology, i.e. the narrow topology on $\cP_p(X)$, is coarser than the $W_p$-topology. We will denote this topology by $\weaker$. Note that this topology is metrisable (for example by the restriction of $d_{LP}$ to $\cP_p(X)$) and we denote the associated metric by $d_\sigma(\cdot,\cdot)$. Given a sequence of probability measures $(\rho_n)_{n \in \N}$ and a topology $\tau$, we denote convergence to a point $\rho \in \cP(X)$ in this topology by $\rho_n \overset{\tau}{\to} \rho$ as $n \to \infty$. For weak convergence in $\cP(X)$, we will often just write $\rho_n \to \rho$ as $n \to \infty$.

We are interested in studying the long time behaviour of $W_p$ gradient flows in $\cP_p(X)$ ($1 < p < \infty$). We start by introducing some preliminary notions from the theory of gradient flows that will play a role in our subsequent analysis. We will use a weaker notion of solution referred to as a \emph{curve of maximal slope} rather than the stronger \emph{gradient flow} notion of solution based on the so-called \emph{Evolutionary Variational Inequality (EVI)}. The advantage of using this notion is that it is softer and allows for non-uniqueness of solutions. For the rest of the paper, even though we will use the two terms interchangeably, we remind the reader that we will only be working with the curve of maximal slope notion of solution.
\begin{defn}(Absolutely continuous curves,~\cite[Definition 1.1.1]{AGS08})
Fix $T>0$ and $1< p<\infty$. We say that that $\rho \in C([0,T]; \cP_p(X))$ is \emph{absolutely continuous}, denoted by $\rho \in AC([0,T]; \cP_p(X))$, if there exists an $m \in \Leb^1([0,T])$, such that
\begin{align}
W_p(\rho(t),\rho(s)) \leq \int_s^t m(r) \dx{r} \, ,
\label{eq:md}
\end{align} 
for all $0 \leq s<t \leq T$. We say that $\rho \in AC([0,\infty);\cP_p(X))$ if $\rho \in AC([0,T]; \cP_p(X))$ for all $T>0$.
\end{defn}  
\begin{defn}(Metric derivative,~\cite[Definition 1.1.2]{AGS08})
Fix $T>0$ and $1< p<\infty$ and consider a curve $\rho \in AC([0,T];\P_p(X)$. Then, the limit
\begin{align}
\abs{\rho'}(t):= \lim_{s \to t} \frac{W_{p}(\rho(t),\rho(s))}{\abs{t-s}} \, ,
\end{align} 
exists $t$ a.e. The function $t \mapsto \abs{\rho'}(t)$ is in $\Leb^1([0,T])$ and is referred to as the \emph{metric derivative} of the curve $\rho$. Furthermore, $\abs{\rho'}$ is the minimal admissible function for the inequality in~\eqref{eq:md}.
\label{def:md}
\end{defn} 
 Consider a proper function $E : \cP_p(X) \to (-\infty,+\infty]$. This function should be thought of as the free energy driving our system. We now present the final ingredient needed to define the notion of solution we will use in this paper, i.e. the weak upper gradient.
 \begin{defn}(Weak upper gradient,~\cite[Definition 1.2.2]{AGS08})
 \label{def:wug}
 A function $G: \cP_p(X) \to [0,+\infty]$ is a \emph{weak upper gradient} for the energy $E$ if for every curve $\rho \in AC([0,T], \cP_p(X))$ such that
 \begin{enumerate}[label=(\roman*)]
\item $G\circ\rho \abs{\rho'} \in \Leb^1([0,T])$
\item $E \circ \rho$ is a.e. equal to a function $\varphi$ with finite pointwise variation in $[0,T]$ and there holds
\[
|\varphi'(t)| \le G(\rho(t)) |\rho'|(t), \quad \text{a.e. }t\in [0,T].
\]
\end{enumerate}
\end{defn}
We introduce the definition of a $p$-curve of maximal slope:
\begin{defn}(Curve of maximal slope,~\cite[Definition 1.3.2]{AGS08})
\label{defn:maxslope}
A curve $\rho \in AC([0,\infty); \cP_p(X))$ is said to be a \emph{$p$-curve of maximal slope} for the functional $E$ with respect to its weak upper gradient $G$ if $E \circ \rho$ is a.e. ~equal to a non-increasing map $\varphi_t = E(\rho(t))$ and
\begin{equation}
\label{eq:maxslope}
\frac{\dx{}}{\dx{t}}\varphi_t= -\frac{1}{p}\abs{\rho'}^p(t) - \frac{1}{q}G^q(\rho(t)) \quad \textrm{a.e. }t \in [0,\infty) \, ,\tag{MS}
\end{equation}
where $q=p/(p-1)$ and $\abs{\rho'}$ is the metric derivative of $\rho$ in the sense of \cref{def:md}. We will often abbreviate `$p$-curve of maximal slope for the functional $E$ with respect to its upper gradient $G$' by \emph{curve of maximal slope}. Furthermore, we will call $\rho \in AC([0,\infty); \cP_p(X))$ a \emph{$p$-curve of maximal slope with initial measure/datum $\rho_0 \in \cP_p(X),E(\rho_0)<\infty$} if $\rho$ is a $p$-curve of maximal slope and $\rho(0)=\rho_0$.
\end{defn}
\begin{remark}
The fact that curves of maximal slope exist under rather mild assumptions on $E$ and $G$ follows from~\cite[Theorems 2.3.1 and 2.3.3]{AGS08}. Furthermore, it is known that under the additional assumption of $\lambda$-convexity (in the sense of McCann~\cite{McC94}) on $E$ and for $p=2$ for every $\rho_0\in\cP_2(X)$, there exists a unique (up to a.e. ~equality) curve of maximal slope $\rho$ with $\rho(0) = \rho_0$ (c.f.~\cite[Theorem 11.1.4]{AGS08}). As mentioned earlier, we will not work with guaranteed uniqueness, instead we choose to fix some curve of maximal slope for which existence is guaranteed; our example in Section~\ref{sec:nonuniq} illustrates how our convergence results work even in the presence of non-uniqueness. 
\end{remark}

\begin{assumption} We will impose the following set of assumptions on our energy, $E$, and its weak upper gradient, $G$:
\begin{enumerate}[label={(A\arabic*)}, ref={A\arabic*}]
\item $E: \cP_p(X) \to (-\infty,+\infty]$  is a proper, bounded below, and lower semicontinuous (l.s.c.) functional with respect to the $\weaker$-topology. \label{ass1}
\item The sublevel set $L_{\leq C}(E):=\set{\rho \in \cP_p(X) : E(\rho) \leq C }$  is relatively compact in the  $\weaker$-topology for all $C \in \R$. \label{ass2} 
\item $G: \cP_p(X) \to [0,\infty]$ is l.s.c. with respect to the $\weaker$-topology. \label{ass3}
\end{enumerate}
\end{assumption}
\begin{assumption} In certain cases, the energy $E$ and its weak upper gradient, $G$, may have extensions $\tilde{E}$ and $\tilde{G}$ defined on $\cP(X)$ such that $\tilde{E}(\rho)=E(\rho)$ and $\tilde{G}(\rho)=G(\rho)$ for all $\rho \in \cP_p(X)$. In this setting, we make the following set of assumptions:
\begin{enumerate}[label={(B\arabic*)}, ref={B\arabic*}]
\item $\tilde{E}: \cP(X) \to (-\infty,+\infty]$ is a proper, bounded below, and l.s.c. functional.\label{ass1'}
\item The sublevel set $L_{\leq C}(\tilde{E}):=\set{\rho \in \cP(X) : \tilde{E}(\rho) \leq C }$  is relatively compact in $\cP(X)$. \label{ass2'}
\item $\tilde{G}: \cP(X) \to [0,\infty]$  is l.s.c. in $\cP(X)$. \label{ass3'} 
\end{enumerate}
\end{assumption}

Note that ~\eqref{ass1'} implies~\eqref{ass1} and ~\eqref{ass3'} implies~\eqref{ass3}. Furthermore,~\eqref{ass2} implies~\eqref{ass2'}.  The set of assumptions~\eqref{ass1'}--\eqref{ass3'} handles precisely the case where gradient flows exist in $\cP_p(X)$ for every $t>0$, while allowing for convergence towards steady states in $\cP(X)$ (but possibly not in $\cP_p(X)$). We refer to the example in~\cref{sec:nonuniq} which illustrates this.    

\begin{assumption}
We may sometimes replace assumptions~\eqref{ass2} and~\eqref{ass2'} by the following assumptions:
\begin{enumerate}[label={(\Alph*2')}, ref={\Alph*2'}]
\item The sublevel set $ L_{\leq C}(E):=\set{\rho \in \cP_p(X) \cap \cY: E(\rho) \leq C }$ is relatively compact in the $\sigma$-topology on $\cP_p(X)$. \label{ass2subset}
\item The sublevel set $ L_{\leq C}(E):=\set{\rho \in \cP_p(X) \cap \cY: E(\rho) \leq C }$ is relatively compact in  $\cP(X)$. \label{ass2'subset} 
\end{enumerate}
Here, $\cY$ is a subset of $\cP_p(X)$ (resp. $\cY \subset \cP(X)$) such that $\bigcup_{t \geq 0} \set{\rho(t)} \subseteq \cY $, where $\rho(t)$ is a $p$-curve of maximal slope of $E$ with respect to $G$ for some initial measure $\rho_0$.  Here, $\cY$ could be a subset which is invariant under the gradient flow (eg. radial measures) or could be the curve of maximal slope itself. 
\end{assumption}

%In fact, equation~\eqref{eq:maxslope} can be relaxed to
%\[
%\frac{\dx{E_t}}{\dx{t}} \leq -\frac{1}{p}|\rho'|^p(t) - \frac{1}{q}G^q(\rho(t)),
%\]
%since the other inequality is guaranteed by the upper gradient property of $G$.

We wish to better understand stationary solutions and convergence towards these states under a dynamical systems viewpoint. One of the main points to clarify is the relation between the set of stationary solutions and the $\omega$-limit set of~\eqref{eq:maxslope} when treated as an abstract dynamical system for a given initial condition. To be precise, we define these notions in our present context. We have adapted some classical definitions from dynamical systems (which can be found in the Appendix~\ref{sec:dynsys}) for our present context to allow for non-unique trajectories. We start with the following result:
\begin{proposition}\label{prop:generates}
Assume that for all $\rho_0 \in \cP_p(X), E(\rho_0) < \infty$ there exists a $p$-curve of maximal slope with initial measure $\rho_0$. We define the set
\begin{align}
Z_{E,p}:= \set{\rho \in \cP_p(X):E(\rho) < \infty} \, ,
\end{align}
which we equip with the $\weaker$-topology.

 Then, the family of mappings  $\set{S_t}_{t \geq 0}$, $S_t:Z_{E,p} \to 2^{Z_{E,p}}$ sending $\rho_0 \mapsto \bigcup_{j \in J}\rho_j(t)$, where $J$ indexes the set of all possible curves of maximal slope with initial measure $\rho_0$,  defines a metric dynamical system in the sense of~\cref{def:mds}.

% Additionally, if for all  $\rho_0 \in \cP_p(X), E(\rho_0) < \infty$ there exists a unique $p$-curve of maximal slope with initial measure $\rho_0$, then there exists a a continuous selection of $\set{S_t}_{t \geq 0}$ which defines a uniquely-defined metric dynamical system in the sense of~\cref{def:mds}.
\end{proposition}
\begin{proof}
We need to check that the conditions of~\cref{def:mds} are satisfied for curves of maximal slope. The condition~\eqref{metdyn1} is trivially satisfied because curves of maximal slope decrease the energy $E$. For condition~\eqref{metdyn2}, one can check that $S_0$ maps any initial datum $\rho_0$ to the initial point of the curve of maximal slope $\rho(0)=\rho_0$. Thus, it coincides with the identity, $I$. Condition~\eqref{metdyn3} is satisfied because curves of maximal slope live in $AC([0,\infty); \cP_p(X))$.
\end{proof}
We now  define what we mean by a stationary state of a curve of maximal slope.
\begin{defn}
\label{defn:equil}
We say that $\mu\in \cP_p(X)$  is a \emph{stationary state} of~\eqref{eq:maxslope} if $E(\mu)<\infty$ and the curve defined for all $t\geq 0$ by $\rho(t):= \mu$ is a curve of maximal slope according to~\cref{defn:maxslope}. We denote the set of stationary states by $\mathcal{E}$.
\end{defn}
Note that~\cref{prop:generates} tells us that curves of maximal slope generate a metric dynamical system on the space $Z_{E,p}$ equipped with the $\weaker$-topology. Thus, the definition of the corresponding $\omega$-limit set should be with respect to this topology. For the convenience of the reader, we recast~\cref{def:wls} into our particular setting.
\begin{defn}
\label{defn:omegalimit}
For $\rho_0\in Z_{E,p}$ and $\rho\in AC([0,\infty); \cP_p(X))$ a curve of maximal slope with $\rho(0) = \rho_0$, we denote by $\omega^\rho(\rho_0)$ the \emph{$\omega$-limit set} with initial data $\rho_0$ subordinate to $\rho$, defined by
\[
\omega^\rho(\rho_0) := \set*{\bar{\rho} \in \cP_p(X) \, : \,  \exists t_n \to \infty \textrm{ s.t. } \rho(t_n)\overset{\sigma}{\to} \bar{\rho} \textrm{ as }n \to \infty}.
\]
\end{defn}
We now give a characterisation of stationary states in relation to the weak upper gradient $G$ which we will use extensively in the sequel.
\begin{lemma}
\label{lemma:equil}
$\rho\in\cP_p(X)$ is a stationary state of~\eqref{eq:maxslope} if and only if $G(\rho)=0$ and $E(\rho)<+\infty$.
\end{lemma}
\begin{proof}
($\implies$) By~\cref{defn:equil}, the curve $\rho(t) = \rho$ solves~\eqref{eq:maxslope}. Since $\rho(t)$ is constant for all times $t$, the first two terms of~\eqref{eq:maxslope} are a.e. ~zero which implies $G(\rho)=G(\rho(t))=0$.

\noindent($\impliedby$) By defining the curve $\rho(t) = \rho$, the first two terms of~\eqref{eq:maxslope} are always zero. Since $G(\rho(t)) = G(\rho) = 0$, equation~\eqref{eq:maxslope} holds automatically.
\end{proof}
The above lemma leads to the following natural relaxation of the notion of stationary state.
\begin{defn} \label{def:wss}
Assume the energy $E$ and the weak upper gradient $G$ have extensions to $\cP(X)$, i.e. there exist  functions $\tilde{E}:\cP(X) \to (-\infty,+\infty]$ and $\tilde{G}:\cP(X) \to [0,+\infty]$, such that $\tilde{E}(\rho)=E(\rho)$ and $\tilde{G}(\rho)=G(\rho)$ for all $\rho \in \cP(X)$. Then, we say that a measure $\mu \in \cP_p(X)$ is a \emph{weak stationary state} of~\eqref{eq:maxslope} if $\tilde{E}(\mu)< +\infty$ and $\tilde{G}(\mu)=0$. We denote the set of weak stationary states by $\mathcal{E}_w$.
\end{defn}
The motivation for the above definition follows from the fact that for certain curves of maximal slope the long time limits may lie outside the set $\cP_p(X)$. This is especially true for curves of maximal slope associated to drift-diffusion type Fokker--Planck equations when the confining potential $V$ does not grow fast enough at $+\infty$. The above definition then allows us to characterise the stationary states of these PDEs as zeroes of the associated extended weak upper gradient $\tilde{G}$ and obtain useful information about the long time behaviour of these curves of maximal slope even if the stationary states lie outside $\cP_p(X)$. We will discuss a specific instance of this in~\cref{sec:fplog}. Accordingly, we give a relaxed definition of $\omega$-limit sets.
\begin{defn}
For $\rho_0\in Z_{E,p}$ and $\rho\in AC([0,\infty);\cP_p(X))$ a curve a maximal slope with $\rho(0)=\rho_0$, we denote by $\omega'^\rho(\rho_0)$ the \textit{relaxed $\omega$-limit set} with initial data $\rho_0$ subordinate to $\rho$, defined by
\[
\omega'^\rho(\rho_0) := \set*{
\bar{\rho}\in \cP(X) \, : \, \exists t_n \to \infty \textrm{ s.t. }\rho(t_n) \to \bar{\rho} \textrm{ in } \cP(X) \textrm{ as } n\to \infty}.
\]
\end{defn}
With these notations, our first main result is the relation between $\omega^\rho(\rho_0)$ (resp. $\omega'^\rho(\rho_0)$) and $\mathcal{E}$ (resp. $\mathcal{E}_w$).
\begin{theorem}
\label{thm:omegastat}
Fix $\rho_0\in Z_{E,p}$ and a curve of maximal slope $\rho \in AC([0,\infty); \cP_p(X))$ with initial condition $\rho(0) = \rho_0$. Under the set of assumptions~\eqref{ass1}, ~\eqref{ass2} (or~\eqref{ass2subset}), ~\eqref{ass3}, we have
\[
\omega^\rho(\rho_0) \subset \mathcal{E},
\]
so that $\omega$-limits are stationary. 
Under the set of assumptions~\eqref{ass1'}, ~\eqref{ass2'} (or~\eqref{ass2'subset}), ~\eqref{ass3'}, we have
$
\omega'^\rho(\rho_0) \subset \mathcal{E}_w,
$
so that relaxed $\omega$-limits are weak stationary states.
\end{theorem}
The set inclusions from~\cref{thm:omegastat} are essential to the following convergence results.
\begin{theorem}
\label{thm:converge}
Assume the energy $E$ and weak upper gradient $G$ satisfy assumptions~\eqref{ass1}, \eqref{ass2} (or~\eqref{ass2subset}), and~\eqref{ass3}. Furthermore, assume that for any $\rho_0 \in Z_{E,p}$ there exists a curve of maximal slope, $\rho \in AC([0,\infty);\cP_p(X))$ of $E$ with respect to $G$. Then, it follows that
\begin{align}
\lim_{t \to \infty}d_{\sigma}(\rho(t),\cE) = \lim_{t \to \infty} \inf_{\mu \in \cE} d_{\sigma}(\rho(t),\mu) =0 \, . 
\end{align}
Furthermore if~\eqref{ass2subset} is satisfied and $\cY$ is closed in the $\sigma$-topology on $\cP_p(X)$, then  
\begin{align}
\lim_{t \to \infty}d_{\sigma}(\rho(t),\cE \cap \cY) =0\, .
\end{align}
\end{theorem}

\begin{theorem}
\label{thm:converge'}
Assume the energy $E$ and the weak upper gradient $G$ have extensions to $\cP(X)$ such that $\tilde{E}$ and weak upper gradient $\tilde{G}$ satisfy assumptions~\eqref{ass1'}, \eqref{ass2'} (or~\eqref{ass2'subset}), and~\eqref{ass3'}. Furthermore, assume that for any $\rho_0 \in Z_{E,p}$ there exists a curve of maximal slope, $\rho \in AC([0,\infty);\cP_p(X))$ of $E$ with respect to $G$. Then, it follows that
\begin{align}
\lim_{t \to \infty}d_{LP}(\rho(t),\cE_w) = \lim_{t \to \infty} \inf_{\mu \in \cE_w} d_{LP}(\rho(t),\mu) =0 \, . 
\end{align}  
Furthermore if~\eqref{ass2'subset} is satisfied and $\cY$ is closed in $\cP(X)$, then
\begin{align}
\lim_{t \to \infty}d_{LP}(\rho(t),\cE_w \cap \cY)=0 \, .
\end{align}
\end{theorem}
We end this section with the following small lemma that can be used to improve the convergence in~\cref{thm:converge,thm:converge'}:
\begin{lemma}
Let $Y$ be a set equipped with two metrics $d_1$ and $d_2$ such that $d_2$ induces a finer topology on $Y$ than $d_1$. Furthermore, let $y:[0,\infty) \to Y$ be a curve such that
\begin{align}
\lim_{t \to \infty} d_1(y(t),Y^*)=\lim_{t \to \infty}\inf_{y \in Y^*} d_1(y(t),Y^*)=0 \, ,
\label{eq:curveconv}
\end{align}
for some $Y^* \subseteq Y$. Assume that $\bigcup_{t \geq 0}\set{y(t)}$ is compact with respect to the topology generated by $d_2$, then
\begin{align}
\lim_{t \to \infty} d_2(y(t),Y^*)=0 \, .
\end{align}
\label{lem:curveconv}
\end{lemma}
\begin{proof}
Let us assume by contradiction that we can find a sequence $(y(t_n))_{n \in \N}$, $t_n \to \infty$, such that
\begin{align}
\liminf_{n \to \infty} d_2(y(t_n),Y^*) \geq \eps \, ,
\end{align} 
for some $\eps>0$. By the compactness of $\bigcup_{t \geq 0}\set{y(t)}$, it follows that there exists a subsequence $(y(t_{n_k}))_{k \in \N}$ and a $y^* \in Y$ such that
\begin{align}
\lim_{k \to \infty}d_2(y(t_{n_k}),y^*) =0 \, .
\end{align}
It follows that
\begin{align}
\lim_{k \to \infty}d_1(y(t_{n_k}),y^*) =0 \, .
\end{align}
This implies, from~\eqref{eq:curveconv}, that $y^* \in Y^*$. This is a contradiction and so the result of the lemma follows.
\end{proof}

\section{Proof of the main abstract result}
\label{sec:omegalimit}
We prove~\cref{thm:omegastat} after a few lemmata. Using this result, we then prove~\cref{thm:converge,thm:converge'}. Throughout this section, we use the same notation as in~\cref{thm:omegastat}; we fix $\rho_0\in Z_{E,p}$ and a curve of maximal slope $\rho\in AC([0,\infty);\cP_p(X))$ with initial condition $\rho(0) = \rho_0$.
\begin{lemma}
\label{lemma:trajrelcpct}
We take the set of assumptions~\eqref{ass1}, ~\eqref{ass2} (or~\eqref{ass2subset}), and ~\eqref{ass3}. We have that the trajectory $\bigcup_{t \geq 0}\set{\rho(t)}$ is relatively compact in $\cP_p(X)$ with respect to the $\sigma$-topology.
\end{lemma}
\begin{proof}
According to equation~\eqref{eq:maxslope}, since the right-hand side is non-positive, we have that $E(\rho(t))$ is a non-increasing function of $t$. By $E(\rho_0)<\infty$, we have that $E(\rho(t)) \leq E(\rho_0) < \infty$ for all $t\geq 0$. Thus,
\begin{align}
\bigcup_{t \geq 0} \set{\rho(t)} \subset \set{\rho \in \cP_p(X) : E(\rho) \leq E(\rho_0) } \, ,
\end{align}
or in the case of~\eqref{ass2subset},
\begin{align}
\bigcup_{t \geq 0} \set{\rho(t)} \subset \set{\rho \in \cP_p(X) \cap \mathcal{Y} : E(\rho) \leq E(\rho_0) } \, .
\end{align}
It follows by~\eqref{ass2} or~\eqref{ass2subset}, that $\bigcup_{t \geq 0}\set{\rho(t)}$ is relatively compact in $\cP_p(X)$ with respect to the $\sigma$-topology.
\end{proof}
~\cref{lemma:trajrelcpct} allows us to extract a time-diverging sequence $t_n\to\infty$ and probability measure $\rho^*\in\cP_p(X)$ such that $\rho(t_n) \overset{\sigma}{\to} \rho^*$ in $\cP_p(X)$. So we have established the
\begin{corollary}
\label{corollary:nonempty}
Under the set of assumptions~\eqref{ass1}, ~\eqref{ass2} (or~\eqref{ass2subset}), ~\eqref{ass3}; we have $\omega^\rho(\rho_0) \neq \emptyset$.
\end{corollary}
~\cref{lemma:trajrelcpct} also has a generalisation to the relaxed $\omega$-limit sets.
\begin{lemma}
\label{lem:relaxedcpct}
We take the set of assumptions~\eqref{ass1'}, ~\eqref{ass2'} (or~\eqref{ass2'subset}), and~\eqref{ass3'}. We have that the trajectory $\bigcup_{t \geq 0}\set{\rho(t)}$ is relatively compact in $\cP(X)$.
\end{lemma}
\begin{proof}
As in the proof of~\cref{lemma:trajrelcpct}, we can use equation~\eqref{eq:maxslope} to deduce the bound $E(\rho(t)) \leq E(\rho_0) < \infty$ for all $t\geq 0$.  
\begin{align}
\bigcup_{t \geq 0} \set{\rho(t)} \subset \set{\rho \in \cP_p(X) : E(\rho) \leq E(\rho_0) } \, ,
\end{align}
or in the case of~\eqref{ass2'subset},
\begin{align}
\bigcup_{t \geq 0} \set{\rho(t)} \subset \set{\rho \in \cP_p(X) \cap \cY : E(\rho) \leq E(\rho_0) } \, .
\end{align}
It follows by~\eqref{ass2'} or~\eqref{ass2'subset}, that $\bigcup_{t \geq 0}\set{\rho(t)}$ is relatively compact in $\cP(X)$.
\end{proof}
\begin{corollary}
\label{corollary:relaxednonempty}
Under the set of assumptions~\eqref{ass1'}, ~\eqref{ass2'} (or~\eqref{ass2'subset}), ~\eqref{ass3'}; we have $\omega'^\rho(\rho_0) \neq \emptyset$.
\end{corollary}
With these basic properties of $\omega^\rho(\rho_0)$ and $\omega'^\rho(\rho_0)$, we now turn to the proof of~\cref{thm:omegastat}.
\begin{proof}[Proof of~\cref{thm:omegastat}]
We only prove the first containment $\omega^\rho(\rho_0) \subset \mathcal{E}$ under the set of assumptions~\eqref{ass1}, ~\eqref{ass2} (or~\eqref{ass2subset}), ~\eqref{ass3}. Fix $\rho^*\in\omega(\rho_0):= \omega^\rho(\rho_0)$ obtained as a limit from some time-diverging sequence $t_n\to \infty$ where $\rho(t_n)\overset{\sigma}{\to} \rho^*$. It suffices to show $G(\rho^*) = 0$ by ~\cref{lemma:equil}. The strategy is to construct a limit curve $\rho^\infty:[0,1]\to\cP_p(X)$ for (a further subsequence of) $\rho(t_n)$ coming from gradient flow theoretic compactness properties of equation~\eqref{eq:maxslope} and then to firstly deduce that $G(\rho^\infty(t)) = 0$ for all $t \in [0,1]$. By showing that  $\rho^\infty(0) = \rho^*$, we conclude the proof.

We define $\rho^n\in AC([0,1];\cP_p(X))$ where $\rho^n(t) = \rho(t_n+t), \, t\in[0,1]$. We will prove compactness of this sequence $\rho^n$ based on a generalised Arzel\`a-Ascoli result. By a similar argument to ~\cref{lemma:trajrelcpct}, for every $t\in[0,1]$, $\overline{\{\rho^n(t)\}_{n\in\mathbb{N}}}^\sigma \subset \{\rho\in\cP_p(X) \, : \, E(\rho) \leq E(\rho_0)\}$ which is compact by assumption~\eqref{ass2} or~\eqref{ass2subset}. Fix $0 \leq s < t \leq 1$. By construction and definition of the metric derivative (cf.~\cref{def:md}), we have for every $n\in\mathbb{N}$
\[
W_p^p(\rho^n(t),\rho^n(s)) = W_p^p(\rho(t_n+t),\rho(t_n+s)) \leq \left(
\int_{t_n+s}^{t_n+t}\abs{\rho'}(u)\frac{\dx{u}}{t-s}
\right)^p(t-s)^p \, .
\]
Applying Jensen's inequality, we estimate further
\[
W_p^p(\rho^n(t),\rho^n(s)) \leq \left(
\int_{t_n+s}^{t_n+t}\abs{\rho'}^p(u)\dx{u}
\right)|t-s|^{p-1} \leq C\abs{t-s}^{p-1} \, ,
\]
where $C$ can be taken to be $\int_0^\infty\abs{\rho'}^p(u)\dx{u} < \infty$ due to the absolute continuity and maximal slope property of the curve and we stress that this constant is independent of $n,s,t$. Finally, we take the $p$-th root of the above inequality to obtain
\begin{equation}
\label{eq:equicty}
W_p(\rho^n(t),\rho^n(s)) \leq C\abs{t-s}^\frac{p-1}{p},
\tag{$\dagger$}
\end{equation}
where we have abused notation by recycling $C$ for the constant. Incidentally, estimate~\eqref{eq:equicty} is precisely the reason we have excluded the values $p=1,\infty$. Applying~\cite[Proposition 3.3.1]{AGS08}, we can find a further subsequence which we omit and a limit curve $\rho^\infty \in C([0,1];\cP_p(X))$ such that pointwise $t\in[0,1]$
\[
\rho(t_n+t)=\rho^n(t) \overset{\sigma}{\to} \rho^\infty(t) \text{ in }\cP_p(X).
\]
It follows then that $\rho^\infty(0)=\rho^*$. Additionally from~\eqref{ass3}, it follows that
\begin{align}
G^q(\rho^\infty(t)) \leq \liminf_{n \to \infty}G^q(\rho^n(t)) \, .
\end{align}
By Fatou's lemma, we then have
\begin{equation}
\label{eq:Gq01}
\int_0^1 G^q(\rho^\infty(t)) \dx{t} \leq \liminf_{n \to \infty}\int_0^1 G^q(\rho^n(t)) \dx{t} \, .
\tag{G$\infty$}
\end{equation}
Returning to equation~\eqref{eq:maxslope}, we integrate in time to obtain, for a.e. $t\in(0,\infty)$,
\[
E(\rho(t))-E(\rho(0)) = -\frac{1}{p}\int_0^t\abs{\rho'}^p(s)\dx{s} - \frac{1}{q}\int_0^tG^q(\rho(s))\dx{s}.
\]
Reversing the sign, applying the lower bound assumption~\eqref{ass1}, and dropping the metric derivative term, we arrive at the bound
\[
\int_0^t G^q(\rho(s))\dx{s} \leq C.
\]
The universal constant $C$ depends only on the lower bound of $E$, the initial value $E(\rho(0))$, and $q$. Therefore, sending $t\to \infty$, we have the integrability of $G^q(\rho(t))$
\begin{equation}
\label{eq:GqL1}
\int_0^\infty G^q(\rho(t))\dx{t} < \infty.
\tag{Gq}
\end{equation}
From the above estimate~\eqref{eq:Gq01} and the full integrability of $G^q(\rho(t))$ in~\eqref{eq:GqL1}, we have that
\begin{align}
\int_0^1 G^q(\rho^\infty(t)) \dx{t} &\leq \liminf_{n \to \infty}\int_0^1 G^q(\rho^n(t)) \dx{t} = \liminf_{n \to \infty}\int_{t_n}^{t_n+1} G^q(\rho(t)) \dx{t} =0 \, .
\end{align}
Since $G^q(\mu) \geq 0$ for all $\mu \in \cP_p(X)$, it follows that $G^q(\rho^\infty(t))=0$ a.e.
$t \in[0,1]$. Thus, there exists a sequence of times $t_m \to 0$ such that $G^q(\rho^\infty(t_m))=0$.
Additionally, since $\rho^\infty \in C([0,1]; \cP_p(X))$, $\rho^\infty(t_m) \overset{\sigma}{\to} \rho^*$ in $\cP_p(X)$ as
$t_m \to 0$. Applying \eqref{ass3} again, we obtain
\begin{align}
G^q(\rho^*) \leq \liminf_{t_m \to 0} G^q(\rho^\infty(t_m))=0 \, .
\end{align}
By~\cref{lemma:equil}, the result follows. The proof of the other containment $\omega'^\rho(\rho_0) \subset \mathcal{E}_w$ under the assumptions~\eqref{ass1'}, ~\eqref{ass2'} (or~\eqref{ass2'subset}), \eqref{ass3'} is exactly the same with appropriate modifications.
\end{proof}
We now have all the ingredients to prove~\cref{thm:converge,thm:converge'}.
\begin{proof}[Proof of~\cref{thm:converge}]
From the first part of~\cref{thm:omegastat}, it suffices to show
\[
\lim_{t\to\infty}d_\sigma(\rho(t),\omega^\rho(\rho_0)) = \lim_{t\to\infty}\inf_{\mu\in\omega^\rho(\rho_0)}d_\sigma(\rho(t),\mu) = 0.
\]
This is essentially identical to the proof of~\cref{thm:ch}. For the convenience of the reader, we redo the proof in our setting.
Assume for a contradiction that this is not true, so there exists a $\delta>0$ and a time-diverging sequence $t_n\to \infty$ as $n\to\infty$ such that
\[
\inf_{\mu\in\omega^\rho(\rho_0)}d_\sigma(\rho(t_n),\mu)) \geq \delta, \quad \textrm{ for all } n\in\mathbb{N}.
\]
Apply~\cref{lemma:trajrelcpct} to find $\rho^*\in\cP_p(X)$ such that, up to subsequence,
\[
d_\sigma(\rho(t_n),\rho^*) \to 0, \quad \text{as } n\to\infty.
\]
This is in contradiction with the previous lower bound since $\rho^*\in \omega^\rho(\rho_0)$. Finally, if~\eqref{ass2subset} is satisfied and $\cY$ is closed on $\cP_p(X)$, it follows that $\omega^\rho(\rho_0) \subset \cY$.
\end{proof}
\begin{proof}[Proof of~\cref{thm:converge'}]
This proof follows the same steps of the previous proof using the second part of~\cref{thm:omegastat}. Here, one needs to replace $d_\sigma, \omega^\rho(\rho_0)$, and~\cref{lemma:trajrelcpct} with $d_{LP}, \omega'^\rho(\rho_0)$, and~\cref{lem:relaxedcpct}, respectively.
\end{proof}
We conclude this section with a result on the limiting behaviour of the energy functional $E$ following similar arguments to~\cite[Theorem 9.2.3]{CH98}. 
\begin{proposition}
For $\rho_0\in Z_{E,p}$ and $\rho\in AC([0,\infty);\cP_p(X))$ the corresponding curve of maximal slope. Assume assumptions assumptions~\eqref{ass1}, \eqref{ass2} (or~\eqref{ass2subset}), and~\eqref{ass3} are satisfied. Then, we have
\begin{tenumerate}
    \item The limit $E_\infty:=\lim\limits_{t\to\infty}E(\rho(t)) = \inf_{t>0}E(\rho(t))$ exists. \label{prop:entlimit!a}
    \item For every $\rho^*\in\omega^\rho(\rho_0)$, we have $E(\rho^*) \leq E_\infty$. In particular, if 
    $$
    E_\infty = \inf_{\rho \in \mathcal{P}_p(X)}E(\rho) \, ,
    $$ 
    then every $\rho^*\in\omega^\rho(\rho_0)$ attains the infimum of $E$ over $\cP_p(X)$.  \label{prop:entlimit!b}
    \item If $E$ is continuous on $\cP_p(X)$ with respect to the $\weaker$-topology, then
    $$
    E(\rho)=E_\infty \, ,
    $$
    for all $\rho \in \omega^\rho(\rho_0)$. \label{prop:entlimit!c}
    \item If $\cE$ is a finite set, then $\omega^\rho(\rho_0)$  is a singleton $\set{\rho^*}$  and 
    $$
    \lim_{t\to \infty} d_\sigma(\rho(t),\rho^*)=0 \,.\label{prop:entlimit!d}
    $$
\end{tenumerate}
\label{prop:entlimit}
\end{proposition}
\begin{proof}
Based on~\cref{eq:maxslope}, we know that $t\mapsto E(\rho(t))$ is a non-increasing function. Based on assumption~\eqref{ass1}, this function is also bounded below so its limit exists which gives~\ref{prop:entlimit!a}.

For~\ref{prop:entlimit!b}, we take a fixed $\rho^*\in \omega^\rho(\rho_0)$, so there exists a time-diverging sequence $t_n \to \infty$ as $n\to \infty$ such that $\rho(t_n) \overset{\sigma}{\to} \rho^*$. By the l.s.c. property in ~\eqref{ass1}, we have
\[
E(\rho^*) \leq \liminf_{n\to\infty}E(\rho(t_n)) = \inf_{t>0}E(\rho(t)).
\]
The proof of~\ref{prop:entlimit!c} is a straightforward consequence of the convergence in~\cref{thm:converge} and the continuity of $E$. For~\ref{prop:entlimit!d}, we note from~\cref{thm:omegastat} that $\omega^\rho(\rho_0) \subset \cE$. Thus, $\omega^\rho(\rho_0)$ is finite and by~\cref{thm:ch} it is connected. It follows that it must be a singleton. The convergence follows from~\cref{thm:converge}. 
% For the convergence of the energy, we consider the curve $\bar{\rho}:[0,1] \to \cP_p(X)$
% \begin{align}
% \bar{\rho}(t)=
% \begin{cases}
%  \rho\bra*{\log \bra*{\frac{1}{1-t}}} & t \in [0,1) \\
% \rho^* &  t=1
% \end{cases}
% \, .
% \end{align}
% We note that $\bar{\rho} \in AC([0,1];\cP_p(X))$. Indeed,
% \begin{align}
% W_p(\bar{\rho}(s),\bar{\rho}(t)) \leq 
% \end{align}   
\end{proof}
\begin{remark}
We note that an analogous result to~\cref{prop:entlimit} can be derived for the relaxed $\omega$-limit set $\omega'^\rho(\rho_0)$ under assumptions~\eqref{ass1'}, \eqref{ass2'} (or~\eqref{ass2'subset}), and~\eqref{ass3'}. 
\end{remark}
\section{Applications to specific PDE}
\label{sec:examples}
This section is devoted to applying our previous results to a wide range of gradient flows. Many of these examples have been studied extensively in the literature; we demonstrate how our abstract theory recovers and even extends some of the results. We emphasise Section~\ref{sec:fplog} for a Fokker-Planck equation with weak confining potential in which the curve belongs to $\cP_p(X)$ for all $t>0$ while its steady state belongs in $\cP(X)\setminus \cP_p(X)$ and Section~\ref{sec:nonuniq} for a pure aggregation equation in which non-unique curves converge to disjoint subsets of the set of stationary states. In Section~\ref{sec:kuramoto}, we examine a one dimensional noisy Kuramoto model with no confining potential although the compactness is compensated by restricting to the torus. Section~\ref{sec:MV} discusses the asymptotics of the McKean-Vlasov equation, that is general interaction potential with linear diffusion to get some asymptotic results even in the presence of phase transitions. We generalise the diffusion term to porous medium type in Section~\ref{sec:aggdiff} with a complete description in the radial symmetric case. In Section~\ref{sec:degdiff}, we turn to $q$-Wasserstein gradient flows for $q\ne 2$ on the torus extending the convergence results of~\cite{Agu03}. In Section~\ref{sec:pme}, the porous medium equation is investigated to recover the well-known result of convergence to the Barenblatt profile. Removing diffusion entirely, we study a pure aggregation equation on $\R$ in Section~\ref{sec:gpy}, providing convergence results using our theory and classical techniques to complement the results of~\cite{GPY17} on the torus.
\subsection{Fokker--Planck equation with weak confinement}
\label{sec:fplog}
In this section, we study the Fokker--Planck equation on $\R^d$ with gradient drift which can be written as
\begin{align}
\begin{cases}
\partial_t \rho = \Delta  \rho + \nabla\cdot(\rho\nabla V) \\
\rho(0)= \rho_0 \in \cP_2(\R^d) \, 
\end{cases}
 ,\label{eq:sfp}
\end{align}
where we choose $V$ to be slowly growing at infinity as
\begin{align}
V(x)= \log (1 + \abs{x}^{2d}) \, ,
\end{align}
see \cite{TV00,AMTU01} for related linear Fokker-Planck equations. Equation~\eqref{eq:sfp} is the forward Kolmogorov equation associated to the following It\^o diffusion process:
\begin{align}
\begin{cases}
\dx{X_t} = -\nabla V(X_t) \dx{t} + \sqrt{2}\dx{B_t} \, \\
\textrm{Law}(X_0) = \rho_0 \in \cP_2(\R^d)
\end{cases}
\end{align}
where $B_t$ is a $d$-dimensional Wiener process. Thus, $\textrm{Law}(X_t)=\rho(t)$, where $\rho(t)$ is the unique weak solution of~\eqref{eq:sfp}. As mentioned in the introduction, following the seminal work in~\cite{JKO98},~\eqref{eq:sfp} can be viewed as gradient flow of the free energy, $E_{FP}: \cP_2(\R^d) \to (-\infty,+\infty]$,
\begin{align}
E_{FP}(\rho)= \int_{\R^d} \rho \log \rho \dx{x} + \int_{\R^d}\log (1+ \abs{x}^{2d}) \dx{\rho}(x) \, ,
\end{align}
with respect to the weak upper gradient $G_{FP}: \cP_2(\R^d) \to (-\infty,+\infty]$,
\begin{align}
G_{FP}(\rho)= \bra*{\int_{\R^d} \abs*{\nabla \log(\rho(1 + \abs{x}^{2d}) )  }^2 \dx{\rho}(x)}^{1/2} \, .
\end{align}
Indeed, we have the following  result
\begin{proposition}
For all $\rho_0 \in Z_{E_{FP},2}$, there exists a unique $2$-curve of maximal slope $\rho \in AC([0,\infty);\cP_2(\R^d))$ of the energy $E_{FP}$ with respect to the weak upper gradient $G_{FP}$ such that $\rho(0)=\rho_0 \in \cP_2(\R^d)$. Furthermore, any such $2$-curve of maximal is equivalent to a weak solution of~\eqref{eq:sfp}.
\label{prop:statfp}
\end{proposition}
\begin{proof}
The proof of this result follows from~\cite[Theorem 11.2.8]{AGS08} and the fact that $V$ is $\lambda$-convex. 
\end{proof}
The reason we have made this choice of the confinement $V$ is because it allows us to demonstrate our abstract convergence results in the setting in which the stationary solutions are not necessarily in $\cP_2(\R^d)$, i.e.~\cref{thm:converge'}. Indeed, we have the following result
\begin{proposition}
Consider the extensions $\tilde{E}_{FP}$  and $\tilde{G}_{FP}$ of $E_{FP}$  and $G_{FP}$ to $\cP(\R^d)$. Then,
\begin{align}
\rho_\infty (\dx{x})= \frac{Z_\infty}{1+\abs{x}^{2d}} \dx{x}\, , \quad Z_\infty= \bra*{\int_{\R^d}\frac{1}{1+ \abs{x}^{2d}} \dx{x}}^{-1} \, ,
\end{align} 
is the unique minimiser of $\tilde{E}_{FP}$ and unique zero of $\tilde{G}_{FP}$ over $\cP(\R^d)$. It follows that it is the unique weak stationary state of the associated curve of maximal slope in the sense of~\cref{def:wss}. Furthermore, $\rho_\infty \in \cP(\R^d)\setminus \cP_2(\R^d)$ and is a stationary weak solution of~\eqref{eq:sfp}.
\end{proposition}
\begin{proof}
That $\rho_\infty$ is the unique minimiser follows from Jensen's inequality. Indeed, we have that for any $\rho \in \cP(\R^d) \cap \Leb^1(\R^d)$
\begin{align}
\tilde{E}_{FP}(\rho)=& \int_{\R^d} \rho \log \rho \dx{x} + \int_{\R^d}\log (1+ \abs{x}^{2d}) \dx{\rho}(x)\\
=& \int_{\R^d} \frac{\rho}{\rho_\infty} \log \bra*{\frac{\rho}{\rho_\infty} }\dx{\rho_\infty} + \log(Z_\infty) \\
\geq & \log(Z_\infty) = \tilde{E}_{FP}(\rho_\infty) \, .
\end{align}
Note that the final inequality is strict unless $\rho=\rho_\infty$. Thus, the result follows. That $\rho_\infty$ is a zero of $\tilde{G}_{FP}$ is apparent by plugging it into the expression for $\tilde{G}_{FP}$. Furthermore, for uniqueness we note that any zero $\rho \in \cP(\R^d) \cap \Leb^1(\R^d)$ of $\tilde{G}_{FP}$ must satisfy 
\begin{align}
\rho(x)= \frac{C_A}{1+ \abs{x}^{2d}} \, ,
\end{align}
for a.e. $ x \in A$, for every connected component $A$ of its support. If $\rho$ is not fully supported then this would imply that $\log(\rho(1+ \abs{x}^{2d}))$ is sum of indicator functions of disjoint sets. Since the indicator function of any set is not in $\SobH^1(\R^d)$ (unless the set is $\R^d$), it follows that
\begin{align}
\tilde{G}_{FP}= +\infty \, .
\end{align}
Thus, $\rho$ must have full support and so is equal to $\rho_\infty$. The fact that $\rho_\infty$ is a weak stationary solution of~\eqref{eq:sfp} follows by plugging it in.
\end{proof}
We can now state and prove our convergence result:
\begin{proposition}
Let $\rho \in AC([0,\infty);\cP_2(\R^d))$ be the unique $2$-curve of maximal slope of the energy $E_{FP}$ with respect to the weak upper gradient $G_{FP}$, for some initial datum $\rho_0 \in Z_{E_{FP},2}$. Then,
\begin{align}
\lim_{t \to \infty} d_{LP}(\rho(t),\rho_\infty)=0 \, ,
\end{align}
where $\rho_\infty \in \cP(\R^d) \cap \Leb^1(\R^d) $ is as given in~\cref{prop:statfp}.
\end{proposition}
\begin{proof}
We check that the conditions of~\cref{thm:converge'} hold true. Note that we have already shown in the proof of~\cref{prop:statfp} that $\tilde{E}_{FP}$ is bounded below and thus it is proper. Furthermore, the fact that it is l.s.c. follows from standard results (cf.~\cite[Lemma 4.3.1]{JLJ98}) and Fatou's lemma. Thus,~\eqref{ass1'} is satisfied. Note that $\tilde{G}_{FP}$ is of the form~\eqref{eq:LSC} and so~\eqref{ass3'} is satisfied.  Furthermore, for any element of a  sublevel set of the energy $\rho \in L_{\leq C }(\tilde{E}_{FP})$, we have that
\begin{align}
\frac{1}{2}\int_{\R^d} \log(1 + \abs{x}^{2d}) \dx{\rho} \leq C- \int_{\R^d} \rho \log \rho \dx{x} - \frac{1}{2} \int_{\R^d}\log (1+ \abs{x}^{2d}) \dx{\rho}(x) \, .
\end{align} 
Using a similar argument as in the proof of~\cref{prop:statfp}, we have that
\begin{align}
\int_{\R^d} \rho \log \rho \dx{x} + \frac{1}{2}  \int_{\R^d}\log (1+ \abs{x}^{2d}) \dx{\rho}(x)  \geq -C_1 \, ,
\end{align}
for some $C_1>0$ independent of $\rho$. It follows that for all $\rho \in L_{\leq C }(\tilde{E}_{FP})$ 
\begin{align}
\frac{1}{2}\int_{\R^d} \log(1 + \abs{x}^{2d}) \dx{\rho} \leq C_2 \, .
\end{align}
Thus, for all $\rho \in L_{\leq C }(\tilde{E}_{FP})$
\begin{align}
\int_{B_R^c(0)}\dx{\rho} \leq \frac{2C_2}{\log(1+ R^{2d})} \, .
\end{align}
By Prokhorov's theorem,~\eqref{ass2'} is satisfied. Thus, by~\cref{thm:converge'},
\begin{align}
\lim_{t \to \infty}d_{LP}(\rho(t),\cE_w) =0 \, ,
\end{align}
where $\cE_w$ is the set of weak stationary states as defined in~\cref{def:wss}. But by~\cref{prop:statfp}, $\cE_w =\set{\rho_\infty}$. Thus, the result follows.
\end{proof}

\subsection{The noisy Kuramoto model}~\label{sec:kuramoto}
In this section, we discuss a special case of the McKean--Vlasov equation  posed on the one-dimensional unit torus, $\T$ (or, equivalently, the unit circle $\mathbb{S}$). This corresponds to the choice of the interaction potential $W_\kappa= - \kappa \cos(2 \pi x), \kappa>0$, which leads to the following PDE:
\begin{align}\label{eq:kuramoto}
\begin{cases}
\partial_t \rho = \partial^2_x \rho + \partial_x(\rho W_\kappa' * \rho)  & (x,t) \in \T \times [0,\infty) \\
\rho(0)= \rho_0 \in \cP_2(\T)  
\end{cases}
\, .
\end{align}
The above equation describes the so-called noisy Kuramoto model for synchronising oscillators. It can be derived as the mean field limit of the following set of interacting SDEs on the torus: 
\begin{align}
\begin{cases}
\dx{X_t^i}= -\dfrac{1}{N}\sum\limits_{j=1}^N W'_\kappa(X_t^i-X_t^j) \dx{t} + \sqrt{2} \dx{B_t^i} \, \\
\mathrm{Law}(X_t^1, \dots, X_t^N)= \rho_0^{\otimes N} \in \cP_2(\T^N) 
\end{cases}
\, ,\label{eq:kurasde}
\end{align}
where the $X_t^i$ represent the phases of the individual oscillators and the $B_t^i$ are independent $\T$-valued Wiener processes. We refer the reader to~\cite{Kur81,SSK88,ABPRS05} for a review of the Kuramoto model and its variants. We note that the Kuramoto model and the dynamics described in~\eqref{eq:kurasde} also correspond to the overdamped Langevin dynamics associated to the so-called Heisenberg $XY$ model from statistical physics for continuous spins on a lattice with mean field interaction (cf.~\cite[Chapter 9]{FV2018}).

As expected, we can think of~\eqref{eq:kuramoto} as a curve of maximal slope of the energy $E_K : \cP_2(\T) \to (-\infty,+\infty]$,
\begin{align}
E_K(\rho):= \int_{\T} \rho \log \rho \dx{x} + \frac{1}{2} \iint_{\T \times \T} W_\kappa(x-y) \dx{\rho}(x) \dx{\rho}(y) \, ,
\end{align}
with respect to the weak upper gradient $G_K :\cP_2(\T) \to (-\infty,+\infty]$,
\begin{align}
G_K(\rho):= \bra*{\int_\T \abs*{\nabla \log  \frac{\rho}{e^{-W_k * \rho}}}^2 \dx{\rho}}^{1/2} \, .
\end{align}
Indeed, we have the following result which we state without proof:
\begin{proposition}\label{prop:comskura}
Let $\rho_0 \in Z_{E_K,2}$. Then, there exists a unique $2$-curve of maximal slope $\rho \in AC([0,\infty); \cP_2(\T))$ of the energy $E_K$ with respect to the weak upper gradient $G_K$  such that $\rho(0)=\rho_0$. Furthermore, curves of maximal slope are equivalent to weak solutions of~\eqref{eq:kuramoto}. 
\end{proposition}
The reason we have made this choice of the interaction potential is because the noisy Kuramoto model is one of the simplest models of this form which exhibits the phenomenon of phase transitions, i.e. a change in structure of stationary solutions and minimisers of the energy as the parameter $\kappa>0$ is varied. Of course the Kuramoto model is not the only model which exhibits this phenomenon. We refer the reader to~\cite{CP10,CGPS20,CG19} where the phenomenon of phase transitions and bifurcations for such equations is discussed in detail in a more general context. We remark that phase transitions also occur in the case of the McKean--Vlasov equation on $\R^d$ discussed in~\cref{sec:MV}, see~\cite{Da83,Tu13,BCCD16,li2019flocking}. We now state without proof the following result which provides a complete characterisation of the set of stationary states and minimisers of the curve of maximal slope associated to the Kuramoto model:
\begin{proposition}{~\cite[Proposition 6.1]{CGPS20}}
Assume $\kappa \leq 2$. Then, there exists a unique minimiser of $E_K$ and stationary state of the associated curve of maximal slope given by $\rho_\infty:= \mathcal{L}_{\T^d}$, the normalised Lebesgue measure on $\T^d$. On the other hand, if $\kappa >2$, there exist two (up to translations) stationary states of the curve of maximal slope given by $\rho_\infty$ and 
\begin{align}
\rho^*(\dx{x}):= Z_\kappa^{-1} e^{\sigma(\kappa)\cos(2 \pi x)} \dx{x}, \qquad Z_\kappa:= \int_\T e^{\sigma(\kappa)\cos(2 \pi x)} \dx{x} \, ,
\end{align}
for some $\sigma: \R_+ \to \R_+$ such that 
$$
\lim_{\kappa \to \infty} \sigma(\kappa)= + \infty \, .
$$
Furthermore, $\rho^*$ is the unique (up to translations) minimiser of $E_K$ for $\kappa>2$.
\label{prop:pt}
\end{proposition}
We are finally in a position to state our convergence result:
\begin{proposition}\label{prop:kuraconverge}
 Let $\rho \in AC([0,\infty);\cP_2(\T))$ be the unique $2$-curve of maximal slope of the energy $E_{K}$ with respect to the weak upper gradient $G_{K}$, for some initial datum $\rho_0 \in \cP_2(\T)$ with $E_{K}(\rho_0)< \infty$. Then, if $\kappa \leq 2$,
 \begin{align}
\lim_{t \to \infty} W_\infty(\rho(t),\rho_\infty)=0 \, ,
\end{align}
where $\rho_\infty(\dx{x})=\dx{x}$. On the other hand if $\kappa >2$ and $E_K(\rho_0)<E_K(\rho_\infty)=0$, then
\begin{align}
\lim_{t \to \infty}W_\infty(\rho(t),T_\rho^*)=0 \, ,
\end{align} 
where $\rho^*$ is as defined in~\cref{prop:pt} and $T_{\rho^*}$ is the set of all translates of $\rho^*$. Furthermore, if $\rho_0$ is symmetric about some $x^* \in \T$, i.e.
\begin{align}
\rho_0(x^*+\cdot)=\rho_0(x^*-\cdot) \, ,
\end{align} 
then
\begin{align}
\lim_{t \to \infty}W_\infty(\rho(t),\rho^*_{x^*})=0 \, ,
\end{align} 
where $\rho^*_{x^*} \in T_{\rho^*}$ is such that
\begin{align}
\int_\T (x -x^*)\rho^*_{x^*} \dx{x}=0 \, .
\end{align}
\end{proposition}
\begin{proof}
That~\eqref{ass1} is satisfied follows from the fact that the energy is finite at $\rho_\infty$ and we have
\begin{align}
E_K(\rho) \geq -\frac{1}{2} \norm{W_\kappa}_{\Leb^\infty(\T)}.
\end{align}
Again, the fact that $E_K$ is l.s.c. follows from standard results. Since $\cP(\T)=\cP_2(\T)$ and $G_K$ is of the form~\eqref{eq:LSC}, it follows that~\eqref{ass3} is satisfied. Finally, since $\cP(\T)$ is compact,~\eqref{ass2} is trivially satisfied. Applying~\cref{thm:converge}, we have that
\begin{align}
\lim_{t \to \infty}d_\sigma(\rho(t), \cE^K)=0 \, ,
\end{align} 
where $\cE^K$ is set of all stationary states of the associated curve of maximal slope. We know from~\cref{prop:pt} that, for $\kappa\leq 2$, $\cE^K= \set{\rho_\infty}$. Thus, for $\kappa\leq 2$,
\begin{align}
\lim_{t \to \infty}d_\sigma(\rho(t),\rho_\infty)=0 \, .
\end{align} 
On the other hand for $\kappa> 2$, we know from~\cref{prop:pt}, that $E_K(\rho^*)< E_K(\rho_\infty)$. Since $t \mapsto E_K(\rho(t))$ is non-increasing, it follows that if $E_K(\rho_0)<E_K(\rho_\infty)=0$ and $\kappa>2$, we have that
\begin{align}
\lim_{t \to \infty}d_\sigma(\rho(t),T_{\rho^*})=0 \, ,
\end{align}
where $T_{\rho^*}$ is the set of all translates of the measure $\rho^*$ defined in~\cref{prop:pt}. We now note that weak solutions of~\eqref{eq:kuramoto} with symmetric initial data about some $x^* \in \T$ remain symmetric about $x^*$ for all $t \geq 0$. This can be seen by performing a change of variables $x \mapsto -x$ in the associated curve of maximal slope~\eqref{eq:maxslope} and noting, from~\cref{prop:comskura}, that curves of maximal slope are unique and equivalent to weak solutions of~\eqref{eq:kuramoto}. Since elements of $T_{\rho^*}$ are all symmetric about some point and since this symmetry is preserved under weak convergence, the convergence to a distinct $\rho^*_{x^*} \in T_{\rho^*}$ follows such that
\begin{align}
\int_\T (x-x^*) \rho^*_{x^*} \dx{x}=0 \, .
\end{align}
We can improve the convergence in $d_\sigma(\cdot,\cdot)$ to any $W_p(\cdot,\cdot)$, $1\leq p <\infty$, by using the fact that $\cP_p(\T^d)$ is compact for all $1 \leq p<\infty$. Since $\rho_\infty$ is bounded below away from zero, we can apply~\cite[Theorem 1.2]{BJR07} to obtain convergence in $W_\infty(\cdot,\cdot)$, thus completing the proof of the result.
\end{proof}
\begin{remark}
We note that in the case of general initial data $\rho_0 \in \cP_2(\T)$ it is non-trivial to identify which distinct element of $T_{\rho^*}$ is seen in the long-time limit. We refer the reader to the discussion in~\cite[Lemma 2.2, Theorem 4.6]{GPP12} in which it is shown that under certain mild assumptions a distinct limit in $T_{\rho^*}$ is selected.
\end{remark}

\subsection{McKean--Vlasov equation on $\R^d$}
\label{sec:MV}
We consider the McKean--Vlasov equation on $\R^d$ as considered in \cite{CMV03,Vi03,CMV06,Tug13,Tug14}. It describes the density of a so-called self-stabilising mean field McKean SDE and is given by
\begin{align}
\begin{cases}
\partial_t \rho = \nabla \cdot \bra*{\rho\nabla(\log \rho + V + W * \rho)}  \\
\rho(0)= \rho_0 \in \cP_2(\R^d) \, 
\end{cases}
\, ,\label{tugaut}
\end{align}
where $V:\R^d \to \R$ is a confining potential and $W:\R^d \to \R$ is an interaction potential satisfying the following assumptions:
\begin{align}
&V \in C^2(\R^d) \textrm{ and  there exists }\lambda > 0 \textrm{ such that } D^2 V (x) \geq \lambda \quad \textrm{ for all } x \notin K \subset \R^d, \textrm{compact},\tag{V1}\label{V1} \\
&\lim_{\abs{x} \to \infty}\norm{D^2 V(x)} = +\infty, \tag{V2}\label{V2} \\
&W \in C^2(\R^d) \textrm{ even and positive}.\tag{W1}\label{W}
\end{align} 
Under these rather minimal assumptions we can apply~\cite[Theorem 11.2.8]{AGS08}  to assert that weak solutions of~\eqref{tugaut} are equivalent to $2$-curves of maximal slope of the energy $E_{MV} : \cP_2(\R^d) \to (-\infty,+\infty]$, given by,
\begin{align}
E_{MV}(\rho):= \int_{\R^d} \rho \log \rho \; \dx{x} + \int_{\R^d}V(x) \dx{\rho}(x) + \frac{1}{2}\int_{\R^d}\bra*{W * \rho}(x) \dx{\rho}(x) \,. 
\end{align}
with respect to the weak upper gradient $G:\cP_2(\R^d) \to [0,\infty]$ which is given by
\begin{align}
G_{MV}(\rho):=\int_{\R^2} \abs*{\nabla \log \frac{\rho}{e^{-V -W * \rho}}}^2 \dx{\rho}(x) \, ,
\end{align}
if the above quantities are finite or as $+\infty$ otherwise. It follows that stationary states of the $2$-curve of maximal slope are equivalent to stationary solutions of~\eqref{tugaut}. To apply our results, we need to check that~\eqref{ass2} and~\eqref{ass3} hold true:
\begin{lemma}
Under the assumptions~\eqref{V1}, \eqref{V2}, and ~\eqref{W}, the energy $E_{MV}$ satisfies~\eqref{ass1} and~\eqref{ass2}. Furthermore, the sublevel set $L_{\leq C}(E_{MV}):=\set{\rho \in \cP_2(\R^d) : E_{MV}(\rho) \leq C }, C \in \R$ is compact in $\cP_2(\R^d)$ with respect to the $W_{2-\eps}(\cdot,\cdot)$ metric.
\label{ass12MV}
\end{lemma}
\begin{proof}
The fact that $E_{MV}$ is proper and l.s.c, i.e.~\eqref{ass1}, follows from standard results (cf.~\cite[Lemma 4.3.1]{JLJ98}) and Fatou's lemma.
In order to prove~\eqref{ass2}, we use the fact that assumptions~\eqref{V1} and~\eqref{V2} imply that 
\begin{align}
\lim_{\abs{x} \to \infty}\bra*{\frac{V(x)}{\abs{x}^2}-1} = +\infty \, . 
\end{align}
Indeed, for all $C>0$ we can find $x_0 \in \R^d$ such that $D^2 V(x) >C$ for all $\abs{x}>\abs{x_0}$, such that

\begin{align}
\lim_{\abs{x} \to \infty}\frac{V(x)}{\abs{x}^2}-1 = +\infty \, . 
\end{align}
Thus, we can find a ball of size $R_0$ such that
\begin{align}
\frac{V(x)}{\abs{x}^2}-1 >1  ,\label{gt1}
\end{align}
for all $x \in B_{R_0}^c$. Now, using~\eqref{W} we have that 
\begin{align}
\frac{1}{2}\int_{\R^d}\bra*{W * \rho}(x) \dx{\rho}(x) \geq 0 \, .
\end{align}
Also, we have the following bound for the entropic term
\begin{align}
\int_{\R^d} \rho \log \rho \dx{x} &= \int_{\R^d} \rho \log \frac{\rho}{e^{-\abs{x}^2}} \dx{x} - \int_{\R^d} \abs{x}^2 \dx{\rho} \\
& \geq C_1 - \int_{\R^d} \abs{x}^2 \dx{\rho} \,  \, ,
\end{align}
where in the last step we simply apply Jensen's inequality. Putting these together we have that any element $\rho$ of a sublevel set of the energy $E_{MV}$, $L_{\leq C}(E_{MV}), C \in \R$ must satisfy
\begin{align}
\int_{\R^d}\bra*{\frac{V(x)}{\abs{x}^2}-1} \abs{x}^2 \dx{\rho}(x) \leq C -C_1 =:C_2 \, .
\end{align}
Fix some $\delta>0$ and consider $R_1 >R_0$. We then have for any $\rho \in L_{\leq C}$ (which is necessarily absolutely continuous) that
\begin{align}
\int_{B_{R_1}}\abs{x}^2 \dx{\rho} =\int_{\R^d} \abs{x}^2 \dx{\rho}- \int_{B_{R_1}^c} \abs{x}^2 \dx{\rho} \, .
\end{align}
Applying~\eqref{gt1}, we have that
\begin{align}
\frac{V(x)}{\abs{x}^2} -1 > \log\bra*{ \frac{\inf_{\partial B_{R_1}} V(x)}{R_1^2} -1 }, \quad \textrm{ for all }x \in B_{R_1}^c \, .
\end{align}
Thus, we obtain 
\begin{align}
\int_{B_{R_1}}\abs{x}^2 \dx{\rho} &=\int_{\R^d} \abs{x}^2 \dx{\rho}- \int_{B_{R_1}^c} \abs{x}^2 \dx{\rho}\\
& >\int_{\R^d} \abs{x}^2 \dx{\rho}- \frac{1}{\log\bra*{ \frac{\inf_{\partial B_{R_1}} V(x)}{R_1^2} -1 }}\int_{B_{R_1}^c}\bra*{ \frac{V(x)}{\abs{x}^2} -1 } \abs{x}^2 \dx{\rho} \\
& \geq \int_{\R^d} \abs{x}^2 \dx{\rho}- \frac{1}{\log\bra*{ \frac{\inf_{\partial B_{R_1}} V(x)}{R_1^2} -1 }} (C-C_1) \, .
\end{align}
Making $R_1$ large enough (independent of the choice of  $\rho \in L_{\leq C}$), we can obtain
\begin{align}
\int_{B_{R_1}}\abs{x}^2 \dx{\rho} \geq  \int_{\R^d} \abs{x}^2 \dx{\rho}- \delta \, .
\end{align}
Thus, we have
\begin{align}
\int_{B_{R_1}^c}\abs{x}^2 \dx{\rho} \leq \delta \, .
\end{align}
Thus, by Prokhorov's theorem, we have compactness of $L_{\leq C}(E_{MV})$ in $\cP_{2-\eps}(\R^d)$ with respect to the metric $W_{2-\eps}(\cdot,\cdot)$. for all $0<\eps\leq1$. Thus,~\eqref{ass2} is satisfied.
\end{proof}

\begin{remark}
Note that the compactness we have derived is stronger than that required by~\eqref{ass2}. Indeed, we have obtained compactness of the sublevel set $L_{\leq C}$ in $\cP_{2-\eps}(\R^d)$ equipped with $W_{2-\eps}$ rather than with respect to the $\sigma$-topology on $\cP_2(\R^d)$.
\end{remark}

\begin{lemma}
Under the assumptions~\eqref{V1}, \eqref{V2}, and ~\eqref{W}, the weak upper gradient $G_{MV}$ satisfies~\eqref{ass3}. 
\label{ass3MV}
\end{lemma}
\begin{proof}
We note that $G_{MV}$ can be trivially extended to a function $\tilde{G}_{MV}: \cP(\R^d) \to [0,\infty]$ such that
$\tilde{G}_{MV}(\rho)=G_{MV}(\rho)$ for all $\rho \in \cP_2(\R^d)$. $\tilde{G}_{MV}$ is of the form~\eqref{eq:LSC} and thus is l.s.c. on $\cP(\R^d)$ by~\cref{thm:LSC}. It follows that $G_{MV}$ is l.s.c with respect to the $\weaker$-topology on $\cP_2(\R^d)$. 
\end{proof}
We conclude this subsection with the following convergence result which recovers the main results of~\cite{Tug13}:
\begin{proposition}
Let $\rho \in AC([0,\infty);\cP_2(\R^d))$ be a $2$-curve of maximal slope of the energy $E_{MV}$ with respect to the weak upper gradient $G_{MV}$, for some initial datum $\rho_0 \in \cP_2(\R^d)$ with $E_{MV}(\rho_0)< \infty$. Then,
\begin{align}
\lim_{t \to \infty} W_{2-\eps}(\rho(t),\cE^{MV})=0 \, ,
\end{align}
for all $0<\eps\leq1$ where $\cE^{MV} \subseteq \cP_2(\R^d)$ is the set of stationary states of the $2$-curve of maximal slope associated to $E_{MV}$ and $G_{MV}$.
\end{proposition}
\begin{proof}
Since by~\cref{ass12MV,ass3MV}, ~\eqref{ass1},~\eqref{ass2}, and~\eqref{ass3} are satisfied, we have by application of~\cref{thm:converge} that
\begin{align}
\lim_{t \to \infty} d_{\sigma}(\rho(t),\cE^{MV})=0 \, .
\end{align}
Using~\cref{lem:curveconv} and the stronger $W_{2-\eps}$ compactness from~\cref{ass12MV}, the result of the proposition follows.
\end{proof}
\begin{remark}
As shown in ~\cite{Da83,Tu13,BCCD16,GP18,li2019flocking} in the case of bistable or multivalleys confinement potentials, McKean-Vlasov equations posed on $\R^d$ of the form discussed in~\eqref{tugaut} also exhibit the phenomenon of phase transitions (cf.~\cref{sec:kuramoto}). A typical example of such a system is provided by the so-called Desai--Zwanzig model. To obtain this model, we set in one dimension
$$
V(x)=x^2- \frac{x^4}{2}
\qquad
\mbox{and} \qquad
W_\kappa(x)=\kappa\dfrac{x^2}{2}
$$
in~\eqref{tugaut}. This system also exhibits a phase transition, i.e. as the value of the parameter $\kappa$ is increased, the system goes from exhibiting one stationary solution to exactly three. We remark that the results we present in~\cref{sec:kuramoto}, especially~\cref{prop:kuraconverge},  hold true analogously in the setting of  the Desai--Zwanzig model or in more complicated models of phase transitions in higher dimensions as discussed in~\cite{Tu13,li2019flocking}. Furthermore, for the Desai--Zwanzig model, we can apply~\Dref{prop:entlimit!c} to argue that the solution converges to a unique stationary state (due to the finiteness of the set of stationary solutions). 
\end{remark}

\subsection{Aggregation-diffusion equation}
\label{sec:aggdiff}
In this section we study the aggregation diffusion equation equation given by
\begin{align}
\partial_t \rho =  \Delta (\rho^m) + \nabla \cdot(\rho \nabla W * \rho) \qquad (t,x) \in (0,\infty) \times \R^d \label{eq:aggdiffRd}
\end{align}
in the regime discussed in~\cite{CHVY19, DYY19}, see \cite{CCY19} for a recent survey. We place the following assumptions on the interaction potential $W$
\begin{align}
&W \in C^\infty(\R^d) \textrm{ is radially symmetric, non-negative, and } W'(r)>0, \textrm{} r>0 \tag{W2} \label{ass:W2} \\
&\textrm{There exists some } C_W>0 \textrm{ such that } W'(r) \leq C_W \tag{W3} \label{ass:W3}  \\
&\textrm{For all } a,b \geq 0 \textrm{ it holds that } W(a+b) \leq C_W (1 + W(1+a) + W(1+b)) \tag{W4}  \label{ass:W4} \\
&W \textrm{ is }\lambda\textrm{-convex} \tag{W5}  \label{ass:W5} \\
& \lim_{r \to \infty} W(r)=+\infty \tag{W6}  \label{ass:W6} \\  
&  \bra*{\lim_{r \to \infty} W(r)}- W \in \Leb^p(\R^d) \textrm{ for some }1 \leq p < \infty \tag{W7}  \label{ass:W7}  
\end{align}
Note that we have modified the assumptions of~\cite{CHVY19, DYY19} so as to ensure that curves of maximal slope exist from~\cite[Theorem 11.2.8]{AGS08}. We have excluded singular interactions so that we do not have to deal with technical difficulties in the gradient flow formulation of~\eqref{eq:aggdiffRd} arising from singularities at the origin. Before proceeding to state and prove our results on convergence for~\eqref{eq:aggdiffRd}, we refer the reader to~\cite{Shu2020} where convergence to the steady state for~\eqref{eq:aggdiffRd} is discussed for $d=1$ and with less restrictive growth assumptions on $W$. 
\begin{proposition}\label{prop:comsaggdiff2}
Let $\rho_0 \in Z_{E_{AD},2}$. Then, under assumptions~\eqref{ass:W2}, \eqref{ass:W3}, and \eqref{ass:W4}, there exists a unique $2$-curve of maximal slope $\rho \in AC([0,\infty); \cP_2(\R^d))$ of the energy
$$
E_{AD}(\rho):= \frac{1}{m-1} \int_{\R^d} \rho(x)^m \dx{x} + \frac{1}{2} \iint_{\R^d \times \R^d} W(x-y) \dx{\rho}(x) \dx{\rho}(y) \, ,
$$
with respect to the weak upper gradient
$$
G_{AD}(\rho):= \bra*{\int_{\R^d}\abs*{\nabla\bra*{\frac{m}{m-1}\rho^{m-1}+ W * \rho }}^2\dx{\rho}(x)}^{1/2} \, ,
$$
 such that $\rho(0)=\rho_0$. Furthermore, curves of maximal slope are equivalent to weak solutions of~\eqref{eq:aggdiffRd}. 
\end{proposition}
\begin{proof}
The proof of this result is an application of~\cite[Theorem 4.20]{DS10}. We just need to check that $W$ satisfies the so-called doubling condition and is $\lambda$-convex. The doubling condition is exactly~\eqref{ass:W4}. The fact that it is $\lambda$-convex is exactly~\eqref{ass:W5}.
\end{proof} 

We now characterise the minimisers of the associated free energy:
\begin{proposition}[Existence and uniqueness of minimisers]\label{prop:exunminRd}
Assume that $m\geq 2$ and that assumptions~\eqref{ass:W2},~\eqref{ass:W3}, \eqref{ass:W4}, and either~\eqref{ass:W6} or~\eqref{ass:W7}, are satisfied. Note that $E_{AD}$ and $G_{AD}$ have extensions to $\cP(\R^d)$ in the sense of~\cref{def:wss}. Then, $\tilde{E}_{AD}$ has a unique (up to translations) minimiser $\rho_* \in \cP(\R^d) \cap \Leb^\infty(\R^d)$ such that $\rho_*$ is compactly supported, radially symmetric and decreasing. Furthermore, $\rho^*$ is the unique weak stationary state of the associated curve of maximal slope.
\end{proposition} 

\begin{proof}
The proof of this result is a combination of results from~\cite{CHVY19,DYY19}. We first apply~\cite[Lemma 3.9]{CHVY19} to argue that all stationary weak solutions of~\eqref{eq:aggdiffRd} are essentially bounded. It then follows from~\cite[Theorem 2.2]{CHVY19} that any stationary solution is compactly supported, radially symmetric, and decreasing (up to a translation). Finally, from~\cite[Theorem 1.1]{DYY19} for $m \geq 2$, we know that stationary solutions are unique up to translations. Since, by the result of~\cite[Theorem 10.4.13]{AGS08}, stationary weak solutions of~\eqref{eq:aggdiffRd} are equivalent to weak stationary states, we have that there exists a unique up to translation weak stationary state, $\rho^* \in \cP(\R^d) \cap \Leb^\infty(\R^d)$,  of the associated curve of maximal slope which compactly supported, radially symmetric, and decreasing. Finally, we argue that $\tilde{E}_{AD}$ has a minimiser. Since any minimiser is a weak stationary state it follows that $\rho^*$ is also the unique minimiser of the free energy. \end{proof}
Motivated by the previous result we introduce the set of all radial probability measures centred at $x_0 \in \R^d$, which is defined as follows
\begin{align}
\cP^{R,x_0}_2(\R^d):= \set*{\rho \in \cP_2(\R^d):\rho(A-x_0)= \rho(\pi(A-x_0)), \forall \pi \in O(d), \forall A\subseteq \R^d \textrm{, measurable}} \, ,
\end{align} 
where $O(d)$ is the group of all $d \times d$ orthogonal matrices. We finally are the position to state our convergence result:
\begin{proposition}
Assume that $m\geq 2$ and that assumptions~\eqref{ass:W2},~\eqref{ass:W3}, \eqref{ass:W4}, and either~\eqref{ass:W6} or~\eqref{ass:W7}, are satisfied. Let $\rho \in AC([0,\infty);\cP_2(\R^d))$ be the unique $2$-curve of maximal slope of the energy $E_{AD}$ with respect to the weak upper gradient $G_{AD}$, for some initial datum $\rho_0 \in \cP_2^{R,x_0}(\R^d)$, $x_0 \in \R^d$, with $E_{MV}(\rho_0)< \infty$. Then,
\begin{align}
\lim_{t \to \infty} d_{\sigma}(\rho(t),\rho^*)=0 \, ,
\end{align}
where $\rho^* \in \Leb^\infty(\R^d) \cap \cP_2^{R,x_0}(\R^d)$ is as given in~\cref{prop:exunminRd}, i.e. it is the unique stationary weak solution of~\eqref{eq:aggdiffRd} and minimiser of $E_{AD}$ with mean $x_0 \in \R^d$, i.e.
\begin{align}
\int_{\R^d} (x-x_0) \dx{\rho^*}(x) =0 \, .
\end{align}
\end{proposition}
\begin{proof}
We first argue that if $\rho \in \cP^{R,x_0}_2(\R^d)$, then
$$
\int_{\R^d}(x-x_0) \dx{\rho}=0 \, .
$$
Indeed, we have that
\begin{align}
\int_{\R^d} (x-x_0) \dx{\rho} =\int_{\R^d } x \dx{(\tau_{-x_0}\rho)} \, ,
\end{align}
where $\tau_x \rho$ is characterised by
\begin{align}
(\tau_x\rho)(A)=\rho(A-x) \, ,
\end{align}
for all measurable $A \subseteq \R^d$. Note that if $\rho \in \cP^{R,x_0}_2(\R^d)$, then $\tau_{-x_0}\rho \in \cP^{R,0}_2(\R^d)$. We thus have that
 \begin{align}
\int_{\R^d} (x-x_0) \dx{\rho} =&\int_{\R^d } x \dx{(\tau_{-x_0}\rho)} \\
=& \int_{H_1^+ } x \dx{(\tau_{-x_0}\rho)} + \int_{H_1^- } x \dx{(\tau_{-x_0}\rho)} + \int_{H_1} x \dx{(\tau_{-x_0}\rho)} \, ,
\end{align}
where $H_i$ is the hyperplane given by
\begin{align}
H_i= \set{x \in \R^d:x_i=0} \, ,
\end{align}
and 
\begin{align}
H_i^+=& \set{x \in \R^d:x_i>0} \\
H_i^-=& \set{x \in \R^d:x_i<0} \, .
\end{align}
Choosing $\pi: \R^d \to \R^d, x \mapsto -x \in O(d) $ and using , we obtain
\begin{align}
\int_{\R^d} (x-x_0) \dx{\rho} =&\int_{H_1^+ } x \dx{(\tau_{-x_0}\rho)} + \int_{H_1^- } x \dx{(\tau_{-x_0}\rho)}(x) + \int_{H_1} x \dx{(\tau_{-x_0}\rho)} \\
=&\int_{H_1^+ } x \dx{(\tau_{-x_0}\rho)} - \int_{H_1^+ } x \dx{(\tau_{-x_0}\rho)} + \int_{H_1} x \dx{(\tau_{-x_0}\rho)} \\
=& \int_{H_1} x \dx{(\tau_{-x_0}\rho)} \, .
\end{align}
If $\rho \in \Leb^1(\R^d)$, the proof is complete since $H_1$ has Lebesgue measure zero. If not we can continue decomposing $H_1$ into $H_1 \cap H_2^+,H_1 \cap H_2^-,H_1 \cap H_2$ and apply the same argument. The argument terminates at some $H_i,i<d$, unless $\rho$ is concentrated at $x_0$ in which case the proof is trivial. Let $\rho \in AC([0,\infty); \cP_2(\R^d))$ be the curve of maximal slope defined in the statement of the proposition. Since $\rho$ is also a weak solution of~\eqref{eq:aggdiffRd}, testing against $x$ in the weak formulation of~\eqref{eq:aggdiffRd} we obtain
\begin{align}
\int_{\R^d}x \dx{\rho(t)} =x_0 \, ,
\end{align}
for all $t \geq 0$.  Furthermore, one can check that the weak formulation of~\eqref{eq:aggdiffRd} is invariant under the action of elements of $\pi \in O(d)$. Thus, if we set $\cY= \cP^{R,x_0}_2(\R^d)$ we have that $\bigcup_{t \geq 0}\set{\rho(t)} \subseteq \cY$. Note that $E_{AD}$ and $G_{AD}$ have extensions to $\cP(\R^d)$ in the sense of~\cref{def:wss}. The fact that $\tilde{E}_{AD}$ is proper and l.s.c. follows from standard results (cf.~\cite[Lemma 4.3.1]{JLJ98}). Thus,~\eqref{ass1'} is satisfied. Furthermore, $\tilde{G}_{AD}$ is of the form specified in~\eqref{eq:LSC} and thus~\eqref{ass3'} is satisfied. Now consider the set
\begin{align}
L_{\leq C}(\tilde{E}_{AD})=\set*{\mu \in \cP_2(\R^d) \cap \cY: \tilde{E}_{AD}(\mu) \leq C } \, .
\end{align}
We now follow the standard argument to obtain compactness in the radial setting as illustrated in~\cite[Theorem 2.1]{CCV15} or~\cite{McC94}. For any $\mu \in L_{\leq C}(\tilde{E}_{AD})$, we have that
\begin{align}
C \geq &\frac{1}{2}\iint_{\R^d \times \R^d} W(x-y) \dx{\mu}(x) \dx{\mu}(y) \\
\geq & \frac{1}{2}\int_{\R^d} \int_{\abs{x+x_0-y} > 1} W(x-y) \dx{\mu}(x) \dx{\mu}(y) + \frac{1}{2}\int_{\R^d} \int_{\abs{x+x_0-y} \leq 1} W(x-y) \dx{\mu}(x) \dx{\mu}(y) \\
\geq & \frac{1}{2}\int_{\R^d} \int_{\abs{x+x_0-y} > 1} W(x-y) \dx{\mu}(x) \dx{\mu}(y) -\frac{1}{2}\norm{W}_{\Leb^\infty(B_{1+ \abs{x_0}}(0))}
\end{align}
We now note that for any $x \in \R^d$ with $\abs{x} \geq 1$, $\set{y \in \R^d: x \cdot (y-x_0) \leq 0} \subseteq \set{y \in \R^d: \abs{x+x_0-y}>1}$. It follows that for every $R \geq 1$ it holds that
\begin{align}
C + \frac{1}{2}\norm{W}_{\Leb^\infty(B_{1+ \abs{x_0}}(0))} \geq \frac{1}{2}\int_{\abs{x} \geq R} \int_{x\cdot (y-x_0) \leq  0} W(x-y) \dx{\mu}(x) \dx{\mu}(y) \, .
\end{align}
Note that $x\cdot (y-x_0) \leq 0$ implies that $\abs{x+x_0-y} \geq\abs{x} $. Since $W$ is radial and monotone increasing, we obtain
\begin{align}
C + \frac{1}{2}\norm{W}_{\Leb^\infty(B_{1+ \abs{x_0}}(0))} \geq& \frac{1}{2}\int_{\abs{x} \geq R} \int_{x\cdot (y-x_0) \leq  0} W(x) \dx{\mu}(x) \dx{\mu}(y)  \\
\geq & \frac{W(R)}{2}\int_{\abs{x} \geq R} \int_{x\cdot (y-x_0) \leq  0}  \dx{\mu}(x) \dx{\mu}(y) \, .
\end{align}
We now use the fact that $\mu \in \cP_2^{R,x_0}(\R^d)$, that $\set{y \in \R^d: x \cdot (y-x_0) \leq 0}$ is a half space passing through $x_0$, and that $\mu$ is not singular since $\tilde{E}_{AD}(\mu)\leq C$ to assert that
\begin{align}
\int_{x\cdot (y-x_0) \leq  0} \dx{\mu}(y)= \frac{1}{2} \, ,
\end{align}
for all $x \in \R^d$. It follows that
\begin{align}
C + \frac{1}{2}\norm{W}_{\Leb^\infty(B_{1+ \abs{x_0}}(0))} \geq\frac{W(R)}{4}\int_{\abs{x} \geq R}   \dx{\mu}(x)   \, .
\end{align}
Since $W(R) \to \infty$ as $R \to \infty$, by Prokhorov's theorem, the set $L_{\leq C}(\tilde{E}_{AD})$ is compact in $\cP(\R^d)$. Thus,~\eqref{ass2'subset} is satisfied. Applying the result of~\cref{thm:converge'} we obtain that
\begin{align}
\lim_{t \to \infty} d_{LP}(\rho(t), \cE_w)=0 \, ,
\end{align}
where $\cE_w$ is the set of weak stationary states of the curve of maximal slope in the sense of~\cref{def:wss}. Note however that by~\cref{prop:exunminRd} it follows that 
\begin{align}
\cE_w \subset \cP_2(\R^d) \, ,
\end{align}
and that $\cE_w$ consists of translates of a single measure $\rho^*$. It follows that
\begin{align}
\lim_{t \to \infty} d_{\sigma}(\rho(t), \cE_w)=0 \, ,
\end{align} 
since $d_\sigma(\cdot,\cdot)$ is just the restriction of $d_{LP}(\cdot,\cdot)$ to $\cP_2(\R^d)$. Even though we cannot extract enough compactness to pass to the limit in the first moment of the curve of maximal slope, we still have enough rigidity to say something about the limit. Let us assume that $\rho(t)$ does not converge to a single measure in $\cE_w$, i.e. there exist subsequences $\rho(t_{n})$ and $\rho(t_{m})$, $t_m,t_n \to \infty$ and measures $\rho_1,\rho_2 \in \cE_w$ such that
\begin{align}
\lim_{n \to \infty}d_\sigma(\rho(t_n),\rho_1) =&0 \\
\lim_{m \to \infty}d_\sigma(\rho(t_m),\rho_2) =&0 \, .
\end{align}
Thus, for any open ball $B_r(x),x \in \R^d$ and $\pi \in O(d)$, we have that
\begin{align}
\int_{B_r(x)} \dx{\rho_1}= \lim_{n \to \infty} \int_{B_r(x)} \dx{\rho(t_n)}=  \lim_{n \to \infty} \int_{\pi(B_r(x-x_0))} \dx{\rho(t_n)}=  \int_{\pi(B_r(x-x_0))} \dx{\rho_1} \, ,
\end{align}
and
\begin{align}
\int_{B_r(x)} \dx{\rho_2}= \lim_{m \to \infty} \int_{B_r(x)} \dx{\rho(t_m)}=  \lim_{m \to \infty} \int_{\pi(B_r(x-x_0))} \dx{\rho(t_m)}=  \int_{\pi(B_r(x-x_0))} \dx{\rho_2} \, .
\end{align}
Here, we have used the fact that $\bigcup_{t \geq 0}\set{\rho(t)} \in \cP_2^{R,x_0}(\R^d)$, the Portmanteau lemma, and the fact $B_r(x)$ and its image under $\pi \in O(d)$ is a continuity set of $\rho_1,\rho_2 \in \cE_w$ (since they are not singular measures). Since open balls generate the Borel $\sigma$-algebra, it follows that both $\rho_1, \rho_2 \in \cP_2^{R,x_0}(\R^d)$. Thus, $\rho_1=\rho_2$. Indeed, they must be translates of each other and if $\rho_2=\tau_{x} \rho_1$ for some $x \neq 0$, then it follows that $\rho_2 \in \cP_2^{R,x_0+x}$ which would be a contradiction. Thus, we have that
\begin{align}
\lim_{t \to \infty}d_{\sigma}(\rho(t),\rho^*) =0 \, ,
\end{align}
where $\rho^* \in \cP_2^{R,x_0}(\R^d)$ and so
\begin{align}
\int_{\R^d}(x-x_0) \dx{\rho^*}=0 \, .
\end{align}
\end{proof}

\subsection{Degenerate diffusion equations on $\T^d$}
\label{sec:degdiff}
Our theory applies to parabolic equations of the form
\begin{align}
\label{eq:dddeq}
\partial_t \rho = \nabla\cdot \left(
\rho \nabla c^* \left(
\nabla F'(\rho)
\right)
\right), \quad \rho(0) = \rho_0 \in \cP_q(\T^d),
\end{align}
where $1 < q < \infty$, $c^*$ denotes the Legendre transform of $c(z) = \frac{1}{q}|z|^q$, the usual cost function in the $q$-Wasserstein metric $W_q$, and $F$ is any of
\[
F(x) = \frac{1}{p-1}x\log x, \quad F(x) = \frac{1}{m(m-1)}x^m, \quad m \geq 1.
\]
Here, $1 < p < \infty$ is the conjugate exponent of $q$. It is by now standard that these choices for $F$ guarantee that the energy functional
\[
E(\rho) = \int_{\mathbb{T}^d}F(\rho)\dx{x}
\]
satisfies the proper, lower bounded, and lower semicontinuity properties of~\eqref{ass1}. Restricting to $\mathbb{T}^d$ guarantees the compactness~\eqref{ass2}. Included in equation~\eqref{eq:dddeq} are the so-called generalized heat equation, parabolic $p$-Laplacian equation~\cite{KV08}, doubly degenerate diffusion equation (see~\cite{SV94} for an $L^p$ theory), and the porous medium equation.
\begin{proposition}
Gradient flow solutions of~\eqref{eq:dddeq} exist provided the initial condition is in the domain of $E, \rho_0 \in D(E)$ and coincide with $q$-curves of maximal slope for the functional $E$ with respect to the weak upper gradient
\[
G(\rho) = \left(\int_{\mathbb{T}^d} |\nabla F'(\rho)|^p \rho\dx{x}\right)^\frac{1}{p},
\]
which is l.s.c. In particular, $G$ satisfies~\eqref{ass3}.
\end{proposition}
\begin{proof}
 We use \cite[Theorems 11.1.3 and 11.3.2]{AGS08} to obtain the equivalence and existence of gradient flow solutions to~\eqref{eq:dddeq} and $q$-curves of maximal slope with respect to the upper gradient $|\partial E|$. It remains to check that $G$ is an upper gradient of $E$. However, this is true using~\cite[Theorem 10.4.6]{AGS08} which shows $G = |\partial E|$ since we are working with absolutely continuous measures.

The convexity of the functionals $E$ is given by~\cite[Proposition 9.3.9]{AGS08} and finally $|\partial E|$ is l.s.c. by~\cite[Corollay 2.4.10]{AGS08}. Alternatively, $G$ is of the form~\eqref{eq:LSC} and hence is l.s.c by~\cref{thm:LSC}.
\end{proof}
\begin{proposition}
Let $\rho \in AC([0,\infty);\cP_q(\T^d))$ be a $q$-curve of maximal slope of the energy $E$ with respect to the weak upper gradient $G$, for some initial datum $\rho_0 \in \cP_q(\T^d)$ with $E(\rho_0)< \infty$. Then,
\begin{align}
\lim_{t \to \infty} W_\infty(\rho(t),\rho_\infty)=0 \, ,
\end{align}
where $\rho_\infty:=\mathcal{L}_{\T^d}$ is the normalized Lebesgue measure on $\T^d$.
\end{proposition}
\begin{proof}
Since we are considering the bounded domain, $\mathbb{T}^d$, the only stationary state is $\rho_\infty$. This can be seen by finding the unique zero of $G$ together with the convexity of the energy. By the previous discussion and proposition, we have that~\eqref{ass1}, \eqref{ass2}, and \eqref{ass3} are satisfied. An application of~\cref{thm:converge} gives
\[
\lim_{t\to \infty} d_\sigma(\rho(t),\rho_\infty) = 0.
\]
Owing to the boundedness of $\T^d$, we can improve this convergence to
\[
\lim_{t\to \infty}W_q(\rho(t), \rho_\infty) = 0.
\]
Furthermore, by~\cite[Theorem 1.2]{BJR07}, since $\rho_\infty$ has a density which is uniformly lower bounded away from zero, we have
\[
\lim_{t\to \infty} W_\infty (\rho(t),\rho_\infty) = 0.
\]
\end{proof}
For more details in a general setting, we refer to~\cite{Agu05,Otto96}. For the asymptotic behaviour in the full space $\mathbb{R}^d$, we lose the compactness; there is no mechanism to guarantee sublevel sets of $E$ are compact in $\mathcal{P}$ in our theory. However, by introducing a change of variables as in~\cite{Agu03}, equation~\eqref{eq:dddeq} can be re-written with a confining potential term, unfortunately breaking the gradient flow structure. There, convergence to an equilibrium solution can be established with exponential rate. We defer an explicit instance of the change of variables to the next example.
\subsection{Porous Medium Equation on $\R^d$}
\label{sec:pme}
In this subsection, we study how our general asymptotic theory can be applied to the porous medium equation on $\R^d$
\[
\partial_t \rho = \Delta \rho^m, \quad m > 1. 
\]
It is known in the vast literature surrounding this equation and its variants (cf.~\cite{CT00,DD02,V03,Vaz07}) that there is a self-similar structure to this equation and in the large time limit tends to a Barenblatt profile. Mathematically, we shall see that seeking a self-similar solution will introduce an equation that has nicer gradient flow properties for which we can apply our theory.Following Otto~\cite{Ott01}, let us introduce the following substitution
\[
\rho(t,x) = \frac{1}{t^{d\alpha}} \hat{\rho}\left(\log t, \frac{x}{t^\alpha}\right), \quad \alpha = \frac{1}{d(m-1)+2}.
\]
Labelling the change of variables as
\[
\tau = \log t, \quad y = \frac{x}{t^\alpha},
\]
we obtain the following evolution equation for $\hat{\rho}$
\[
\partial_\tau\hat{\rho} = \Delta_y \hat{\rho}^m + \alpha\nabla_y\cdot(y\hat{\rho}).
\]
This is a gradient flow with energy functional
\[
E(\hat{\rho}) = F_m(\hat{\rho}) + \frac{\alpha}{2}\int_{\R^d} |y|^2 \dx{\hat{\rho}},
\]
where the internal energy $F_m$ is given by
\[
F_m(\hat{\rho}) = \frac{1}{m-1}\int_{\mathbb{R}^d} \hat{\rho}^m \dx{x}.
\]
Arguing similarly to the previous section,  this functional and its associated dissipation
\[
G^2(\hat{\rho}) = \int_{\R^d} \hat{\rho}\left|
 \frac{m}{m-1}\nabla \hat{\rho}^{m-1} + \alpha y 
\right|^2\dx{y},
\]
fulfil~\eqref{ass1}, \eqref{ass2}, and \eqref{ass3}. Here, we note that the compactness on $\mathbb{R}^d$ is owed to the confining potential for the second moment of $\hat{\rho}$ in $E$. Hence, after a self-similar scaling, the porous medium equation fits into our abstract framework. Again, the lower semicontinuity of $G$ is given by~\cref{thm:LSC}. Before stating the convergence result, we repeat the well known fact
\begin{lemma}[\cite{Ott01}]
\label{lem:barenblatt}
The unique stationary state in the evolution of $\hat{\rho}$ is
\[
\hat{\rho}_\infty(y) = \left(
C - \frac{\alpha(m-1)}{2m}|y|^2
\right)_+^\frac{1}{m-1},
\]
where $C>0$ is a constant such that $\hat{\rho}_\infty$ has unit mass.
\end{lemma}
\begin{proof}[Sketch proof]
We only show here the formal computation that determines $\hat{\rho}_\infty$. For the uniqueness of the stationary state, we refer to~\cite{Ott01} for more details. The zeroes of $G^2(\hat{\rho})$ are precisely given by
\[
0 = \alpha y + \frac{m}{m-1}\nabla  \hat{\rho}^{m-1},
\]
on the support of $\hat{\rho}$. Integrating out $y$ gives, for some constant $C_1$,
\[
C_1   = \frac{\alpha}{2}|y|^2 + \frac{m}{m-1}\hat{\rho}^{m-1}.
\]
The formula is proven noting the positive part has to be taken since these computations are on the support of $\hat{\rho}_\infty$.
\end{proof}
\begin{proposition}
Let $\hat{\rho} \in AC([0,\infty);\cP_2(\R^d))$ be a $2$-curve of maximal slope of the energy $E$ with respect to the weak upper gradient $G$, for some initial datum $\rho_0 \in Z_{E,2}$. Then, for every $0 < \eps < 1$
\begin{align}
\lim_{\tau \to \infty} W_{2-\eps}(\hat{\rho}(\tau),\hat{\rho}_\infty)=0.
\end{align}
\end{proposition}
\begin{proof}
In the introductory discussion,~\eqref{ass1}, \eqref{ass2}, and \eqref{ass3} are all satisfied so we can apply~\cref{thm:converge}. Furthermore, the previous ~\cref{lem:barenblatt} shows that the set of stationary states is the singleton $\{\hat{\rho}_\infty\}$ so that $d_\sigma (\hat{\rho}(\tau), \hat{\rho}_\infty) \to 0$ as $\tau \to \infty$. We can upgrade the convergence to $W_{2-\eps}$ since the compactness of sublevel sets of the energy $\{E \le C\}$ holds up to $W_{2-\eps}$; second moment control is built into $E$.
\end{proof}
\subsection{Nonlocal interaction equations: a model for consensus convergence}
\label{sec:gpy}
In this example, we study pure interaction in one dimension which is still quite complicated. One important example arises in the family of Hegselmann-Krause models for opinion dynamics~\cite{HK02}. It can be written as the following aggregation equation on $\mathbb{R}$
\begin{align}
\label{eq:GPY}
\begin{cases}
\partial_t \rho = \partial_x (\rho \partial_x(\psi * \rho))  \\
\rho(0)= \rho_0 \in \cP_2(\R) \, 
\end{cases}
\, ,
\end{align}
where the interaction kernel $\psi :\mathbb{R} \to \mathbb{R}$ is given as a primitive
\[
\psi(x) = \int_{-\infty}^x y\phi(y)\dx{y},
\]
for a compactly supported even function $\phi : \R \to \mathbb{R}_{\geq 0}$. We refer to~\cite{LT04,BV06,BL09,BCL09,FR10,FR11,CDFLS11,BCDFP15} for more details on aggregation equations and finite time blowup of $L^1$-solutions. For simplicity, we will assume $\phi$ is qualitatively similar but more regular than the examples considered in~\cite{GPY17}. We assume that $\phi \in C_c^\infty(\R)$ so that the resultant $\psi$ has the following properties;
\begin{enumerate}
	\item $\psi \in C_c^\infty(\R)$, 
	\item $\supp (\psi) \subset \supp (\phi)$,
	\item and $\psi$ is non-positive and even.
\end{enumerate} 
In~\cite{GPY17}, the examples of $\phi$ considered are all $L^\infty$ which only guarantees almost everywhere Lipschitz continuity of $\psi$. As mentioned previously, we require the further regularity assumptions for simplicity.
% \begin{figure}[H]
% 	\begin{tikzpicture}
% 	\draw[thick,->] (-8,0) -- (4,0) node[anchor=west] {$x$};
% 	\draw[thick,->] (-2,-6) -- (-2,6) node[anchor=south] {$y$};
% 	\draw (-6,0) .. controls (-4,0) and (-3.5,-2.75) .. (-2,-2.9) node[anchor=north west] {$\inf \psi$} .. controls (-0.5,-2.75) and (0,0) .. (2,0);
% 	\node[circle,draw] at (-7,5) {$\psi$};
% 	\node[anchor = south west] at (-2,0) {$(0,0)$};
% 	\end{tikzpicture}
% 	\caption{Typical profile of $\psi$}
% 	\label{fig:psigpy}
% \end{figure}
The associated energy functional and candidate weak upper gradient are
\[
E_{GPY}(\rho) := \frac{1}{2}\int_\R [\psi * \rho] (x) \dx{\rho}(x), \qquad G_{GPY}^2(\rho) := \frac{1}{2} \int_\R \left|
\partial_x (\psi * \rho)
\right|^2 \dx{\rho}(x).
\]
Since $\psi$ is $C_c^\infty$, one can check that the regularity and $\lambda$-convexity assumptions of~\cite[Assumptions NL0-3]{CDFLS11} all hold and there exists a unique 2-curve of maximal slope for~\eqref{eq:GPY} with $G_{GPY}$ being the upper gradient of $E_{GPY}$. Moreover, weak solutions of~\eqref{eq:GPY} are equivalent to the associated curves of maximal slope.

We explain below that all the assumptions of our abstract theory are satisfied. However, in the sequel we verify some Dobrushin type estimates to obtain rates of convergence. We start with the following soft convergence result:
\begin{proposition}
\label{prop:GPYabconv}
Let $\rho \in AC([0,\infty);\cP_2(\R))$ be a $2$-curve of maximal slope of the energy $E_{GPY}$ with respect to the weak upper gradient $G_{GPY}$, for some initial datum $\rho_0 \in Z_{E_{GPY},2}$. Then,
\begin{align}
\lim_{t \to \infty} d_\sigma\left(\rho(t),\cE_{GPY}\right)=0 \, ,
\end{align}
where $\cE_{GPY} \subseteq \cP_2(\R)$ is the set of stationary states of the $2$-curve of maximal slope associated to $E_{GPY}$ and $G_{GPY}$.
\end{proposition}
\begin{proof}
The listed properties of $\psi$ guarantee that $E_{GPY}$ satisfies~\eqref{ass1}, (indeed it is continuous on $\cP_2(\R)$). Furthermore, $G_{GPY}$ has an extension to $\cP(\R)$ which is exactly of the form~\eqref{eq:LSC}. Thus, by~\cref{thm:LSC},~\eqref{ass3} also holds. We verify~\eqref{ass2subset} in the sequel with~\cref{prop:GPYrelcpct} and then apply~\cref{thm:converge}.
\end{proof} 
It is natural to ask if there is a characterisation for the stationary states for~\eqref{eq:GPY}. The situation is complicated even in one dimension. A priori, we can generate an entire family of linear combination of Diracs belonging to the set of stationary states. Suppose the support of $\psi$ is contained in $[-R,R]$, $\{x_i\}_{i\in I} \subset \R$ is a discrete set of points (finite or infinite), and $\{m_i\}_{i\in I} \subset \R_{>0}$ is indexed the same such that
\[
\sum_{i\in I}m_i = 1, \quad \sum_{i\in I}m_i x_i^2 < \infty.
\]
If we furthermore assume $\inf_{i\neq j\in I}|x_i - x_j| > 2R$, then the combination $\mu = \sum_{i\in I}m_i \delta_{x_i}$ is a stationary state for~\eqref{eq:GPY} since each particle $x_i$ has zero interaction with different particles $x_j$. We are not aware if all stationary states to~\eqref{eq:GPY} are combinations of Diracs described in this way. If instead we assume $\psi$ is analytic (hence, not compactly supported) and satisfies a growth condition, there is a characterisation of stationary states of~\eqref{eq:GPY} as a finite sum of Diracs~\cite[Proposition 2.2]{FR10}. 
\begin{lemma}[Second moment estimate]
	\label{lem:GPY2ndmomest}
For $\displaystyle \rho_0\in Z_{E_{GPY},2}$, let $\rho\in AC([0,\infty);\cP_2(\R))$ be the associated 2-curve of maximal slope for $E_{GPY}$ with respect to the upper gradient $G_{GPY}$. Then, we have the estimate
\begin{align}
\frac{\dx{}}{\dx{t}}\int |x|^2 \dx{\rho}_t(x) \leq 0 \, ,
\end{align}
for all $t \geq 0$.
\end{lemma}
\begin{proof}
% <<<<<<< HEAD
% By AGS (\textcolor{blue}{again, need to double-check this because Jose, Dejan, \dots}), $\rho : t \mapsto \rho_t$ can be thought of as a solution to~\eqref{eq:GPY} in the sense of distributions. Formally, we test equation~\eqref{eq:GPY} against $|x|^2/2$ yielding
% \begin{align}
% \frac{\dx{}}{\dx{t}}\int \frac{|x|^2}{2}\dx{\rho}(x) 	&= - \int x [(\psi')*\rho](x) \dx{\rho}(x) 	\\
% &= - \int x \int (x-y) \phi(x-y) \dx{\rho}(y) \dx{\rho}(x) 	\\
% &= -\frac{1}{2}\iint |x-y|^2 \phi(x-y)\dx{\rho}(y) \dx{\rho}(x) \leq 0.
% \end{align}
% Here, we have used, $\psi'(x) = x\phi(x)$, symmetry in swapping variables $x\leftrightarrow y$, and the positivity of $\phi$ to arrive at the conclusion.

Based on~\cite{CDFLS11}, $\rho : t \mapsto \rho_t$ can be thought of as a solution to~\eqref{eq:GPY} in the sense of distributions. Formally, we test equation~\eqref{eq:GPY} against $|x|^2/2$ yielding
\begin{align}
\frac{\dx{}}{\dx{t}}\int \frac{|x|^2}{2}\dx{\rho}(x) 	&= - \int x [(\psi')*\rho](x) \dx{\rho}(x) 	\\
&= - \int x \int (x-y) \phi(x-y) \dx{\rho}(y) \dx{\rho}(x) 	\\
&= -\frac{1}{2}\iint |x-y|^2 \phi(x-y)\dx{\rho}(y) \dx{\rho}(x) \leq 0.
\end{align}
Here, we have used $\psi'(x) = x\phi(x)$, symmetry in swapping variables $x\leftrightarrow y$, and the positivity of $\phi$ to arrive at the conclusion. To make this argument rigorous, one can use a standard cut-off approximation of $|x|^2/2$.
\end{proof}
\begin{proposition}[Compactness of trajectories]
\label{prop:GPYrelcpct}
For $\rho_0\in Z_{E_{GPY},2}$, let $\rho\in AC([0,\infty); \cP_2(\R))$ be the associated 2-curve of maximal slope for $E_{GPY}$ with respect to the upper gradient $G_{GPY}$. We have that $\{\rho_t\}_{t\in[0,\infty)}$ is relatively compact in $\cP(\R)$ with limits in $\cP_2(\R)$. In particular,~\eqref{ass2subset} holds.
\end{proposition}
\begin{proof}
Our strategy for this proof will be to first show that the family $\{\rho_t\}_{t\in[0,\infty)}$ is tight in $\mathcal{P}(\R)$ using the previous ~\cref{lem:GPY2ndmomest}. The gain of moments will then also come from ~\cref{lem:GPY2ndmomest}.

\noindent \textit{Step 1: $\{\rho_t\}_{t\in[0,\infty)}$ is tight in $\mathcal{P}(\R)$}

Set $E_0:= \int |x|^2 \dx{\rho}_0(x) \geq \int |x|^2 \dx{\rho}_t(x)$ where the inequality comes from ~\cref{lem:GPY2ndmomest}. Fix $\eps>0$ and choose $R>0$ large enough such that $E_0/R^2 < \eps$. For every $t\geq 0$, we have the following chain of inequalities;
\[
\int_{B_R} \dx{\rho}_t(x) = 1 - \int_{\R\setminus B_R} \dx{\rho}_t(x) \geq 1 - \frac{1}{R^2}\int_{\R\setminus B_R} |x|^2\dx{\rho}_t(x) \geq 1 - \frac{E_0}{R^2} \geq 1-\eps.
\]
This establishes tightness in $\mathcal{P}(\R)$.

\noindent \textit{Step 2: Limits belong to $\cP_2(\R)$}

From the previous step, we can find $\bar{\rho}\in\mathcal{P}(\R)$ and a time-divergent sequence $(t_n)_{n\in\mathbb{N}}$ such that
\[
\rho_{t_n}\weakstar \bar{\rho} \quad \text{in }\mathcal{P}(\R).
\]
We use Fatou's lemma and lemma~\ref{lem:GPY2ndmomest} to give
\[
\int |x|^2 \dx{\bar{\rho}(x)} \leq \liminf_{n \to \infty}\int |x|^2 \dx{\rho}_{t_n}(x) \leq E_0.
\]
\end{proof}
At this point, we may conclude that our abstract theory applies and use the previous results to deduce convergence towards stationary states (which exist by Corollary~\ref{corollary:nonempty}). In fact, we may improve the convergence of ~\cref{prop:GPYabconv} to $W_{2-\eps}$ or, in the case of compactly supported initial data, $W_\infty$. We can provide a more concrete description with classical methods. 
\begin{proposition}[Sum of Diracs]
	\label{ex:particle}
Consider $\phi\in C_c^\infty(\R)$ a non-negative and even function described previously but now with the specific conditions
\[
\supp(\phi) = [-R,R], \quad \phi > 0 \text{ on } (-R,R).
\]
Let $N\in\mathbb{N}$ be fixed with weights $\omega^i\in[0,1]$ and initial positions $x_0^i \in \mathbb{R}$ such that
\[
\sum_{i=1}^N\omega^i = 1, \quad \sup_{i,j=1,\dots,N} |x_0^i - x_0^j| =: D < R.
\]
Then, for $\rho_0 = \sum_{i=1}^N \omega^i \delta_{x_0^i}$ as an initial condition to equation~\eqref{eq:GPY}, the corresponding solution $\rho$ converges weakly and in 1-Wasserstein exponentially to $\delta_{\sum_{i=1}^N \omega^i x_0^i}$ as $t\to\infty$. The rate depends on $\phi$ and $D$.
\end{proposition}
We will prove this result after a few key lemmata.
\begin{lemma}[Evolution of particles]
\label{lem:particleevolution}
Assume the same notation and setting of ~\cref{ex:particle}. Then, $\rho_t = \sum_{i=1}^N \omega^i \delta_{x_t^i}$ is the solution to~\eqref{eq:GPY} where each particle evolves according to
\begin{equation}
\label{eq:particleevolution}
\frac{\dx{}}{\dx{t}}x^i = - \sum_{j=1}^N\omega^j\psi'(x^i-x^j), \quad x^i(0) = x_0^i.
\end{equation}
\end{lemma}
\begin{proof}
The right-hand side of~\eqref{eq:particleevolution} is smooth and compactly supported in each $x^i$, hence classical Cauchy-Lipschitz theory guarantees existence and uniqueness of solutions for each $x^i$ for all times. For any test function $\tau$, we have
\begin{align}
&\quad \frac{\dx{}}{\dx{t}}\int_{-\infty}^\infty \tau(x)\dx{\rho}_t(x) 	= \frac{\dx{}}{\dx{t}}\sum_{i=1}^N \omega^i\tau(x^i) = \sum_{i=1}^N \omega^i \tau'(x^i)\frac{\dx{}}{\dx{t}}x^i = - \sum_{i=1}^N \omega^i\tau'(x^i)\sum_{j=1}^N \omega^j\psi'(x^i-x^j) 	\\
&= - \sum_{i=1}^N\omega^i\tau'(x^i) (\psi'*\rho_t)(x^i) = - \int_{-\infty}^\infty \tau'(x)\partial_x(\psi * \rho_t)(x) \dx{\rho}_t(x).
\end{align}
This is the weak formulation of~\eqref{eq:GPY} hence $\rho_t$ defined in the statement of the lemma is a solution. The interaction $\psi'$ has the necessary regularity to apply standard results that can be found in~\cite{Gol16} to deduce uniqueness.
\end{proof}
We turn to studying properties of the dynamical system given in~\eqref{eq:particleevolution} since these translate into crucial statements at the level of the PDE~\eqref{eq:GPY}. Using an argument from~\cite{MPPD19}, we have
\begin{lemma}[Particles do not collide in finite time]
\label{lem:nocollide}
Assume the same notation and setting of ~\cref{ex:particle}. For $i=1,\dots,N$ take the unique solutions $x^i$ to the ODE system~\eqref{eq:particleevolution}. There exists a constant $K>0$ depending only on $\psi$ such that
\[
|x^i(t) - x^j(t)| \geq |x_0^i-x_0^j|e^{-Kt} \, ,
\]
for all $i,j=1,\dots,N$  and all  $t >0$.
\end{lemma}
\begin{proof}
We study the evolution of $|x^i-x^j|^2$ which yields
\begin{align}
\frac{\dx{}}{\dx{t}}|x^i-x^j|^2 	&= 2(x^i-x^j)\left(\frac{\dx{}}{\dx{t}}x^i-\frac{\dx{}}{\dx{t}}x^j\right) 	= - 2(x^i-x^j)\sum_{k=1}^N\omega^k[\psi'(x^i-x^k) - \psi'(x^j-x^k)] 	\\
&\geq -2|x^i-x^j|\sum_{k=1}^N\omega^k|\psi'(x^i-x^k) - \psi'(x^j-x^k)| 	\\
&\geq -2|x^i-x^j|^2 \sum_{k=1}^N\omega^k \sup_{\xi\in\mathbb{R}}|\psi''(\xi-x^k)| 	\\
&\geq -2K|x^i-x^j|^2.
\end{align}
The second line comes from Cauchy-Schwarz inequality. The third line makes use of a Mean-Value type inequality. Finally, the fourth line comes from using the smoothness and compact support of $\psi$ to produce the constant $K$. We conclude after applying Gr\"onwall's inequality on the differential inequality.
\end{proof}
One important consequence of this lemma is the following corollary.
\begin{corollary}[Contracting particle cloud]
	\label{cor:contractparticle}
	Assume the same notation and setting of~\cref{ex:particle}. For $i=1,\dots,N$ take the unique solutions $x^i$ to the ODE system~\eqref{eq:particleevolution}. Assume each initial particle is unique so that $x_0^i\neq x_0^j, \, $ for all $ i\neq j$. Let $k,l$ denote the indices of the right-most and left-most particles at time zero, respectively. In other words,
	\[
	x_0^k = \max_{i=1,\dots,N}x_0^i, \quad x_0^l = \min_{i=1,\dots,N}x_0^i.
	\]
	Then, these indices still mark the right-most and left-most particles at all future times, meaning
	\[
	x^k(t) = \max_{i=1,\dots,N}x^i(t), \quad x^l(t) = \min_{i=1,\dots,N}x^i(t), \quad \textrm{ for all } t\geq 0.
	\]
	Furthermore, the particle cloud is contracting in the sense that
	\[
	\frac{\dx{}}{\dx{t}}x^k < 0, \quad \frac{\dx{}}{\dx{t}}x^l > 0, \quad \textrm{ for all } t>0.
	\]
\end{corollary}
\begin{proof}
From ~\cref{lem:nocollide}, we immediately establish the first part of the result saying
\[
x^k(t) = \max_{i=1,\dots,N}x^i(t), \quad x^l(t) = \min_{i=1,\dots,N}x^i(t), \quad \textrm{ for all } t\geq 0.	
\]
For the second part, we first claim that the particle cloud remains in $(-R,R)$, that is
\begin{equation}
\label{eq:cloudinR}
\sup_{i,j=1,\dots,N} |x^i(t)-x^j(t)| < R, \quad \textrm{ for all } t\geq0.
\end{equation}
We prove this claim by re-writing the particle cloud diameter in terms of $x^k-x^l$, indeed ~\cref{lem:nocollide} guarantees
\[
\sup_{i,j=1,\dots,N} |x^i(t)-x^j(t)| = x^k(t) - x^l(t).
\]
Now, we study the evolution of the right-hand side.
\[
\frac{\dx{}}{\dx{t}}(x^k-x^l) = -\sum_{j=1}^N\omega^j(\psi'(x^k-x^j)-\psi'(x^l-x^j)) \leq 0.
\]
The inequality comes from the following. By definition, $x^k-x^j > 0$ and $x^l - x^j <0$ for every index $j=1,\dots,N$. Recalling that $\phi\geq0$ and $\psi'(x) = x\phi(x)$, we have that
\[
\psi'(x^l-x^j) \leq 0 \leq \psi'(x^k-x^j), \quad \textrm{ for all } j=1,\dots,N.
\]
Therefore, the claimed inequality~\eqref{eq:cloudinR} holds since $\sup_{i,j=1,\dots,N} |x_0^i-x_0^j| = D < R$. 

We return to $\frac{\dx{}}{\dx{t}}x^k$ and remark that the following computations are easily extended to $\frac{\dx{}}{\dx{t}}x^l$. Observe
\[
\frac{\dx{}}{\dx{t}}x^k = - \sum_{j=1}^N\omega^j \psi'(x^k-x^j) = -\sum_{j=1}^N \omega^j (x^k-x^j)\phi(x^k-x^j) < 0, \quad \textrm{ for all } t\geq0.
\]
Here, we once again use the property that $x^k-x^j > 0$ for all times $t\geq0$. Furthermore, the claimed inequality~\eqref{eq:cloudinR} guarantees $\phi(x^k-x^j) > 0$ by our assumption on $\phi$ in~\cref{ex:particle}.
\end{proof}
We are now in a position to prove ~\cref{ex:particle}.
\begin{proof}[Proof of ~\cref{ex:particle}]
The proof is broken down into a few steps. To fix ideas, we will take the $x^i$ satisfying~\eqref{eq:particleevolution} from ~\cref{lem:particleevolution} and the indices $k,l$ denoting the right-most and left-most particles, respectively. Recall that these indices are well-defined and persist for all time by Corollary~\ref{cor:contractparticle}. This strategy is similar to the argument in~\cite{CDFLS11}.
\begin{enumerate}
	\item Establish the stationarity of the centre of mass $\sum_{j=1}^N \omega^j x_t^j = \sum_{j=1}^N \omega^j x_0^j$.
	\item Show that each $x^i$ converges with exponential rate to the centre of mass $\sum_{j=1}^N\omega^jx_0^j$.
	\item Convert the convergence at the particle level to convergence at the PDE level.
\end{enumerate}
\noindent \textit{Step 1:} We compute
\begin{align}
&\quad \frac{\dx{}}{\dx{t}}\left(
\sum_{j=1}^N\omega^jx_t^j
\right) = -\sum_{j=1}^N\omega^j \sum_{i=1}^N\omega^i\psi'(x^j-x^i) 	\\ &= -\frac{1}{2}\sum_{j=1}^N\sum_{i=1}^N\omega^j\omega^i\psi'(x^j-x^i) - \frac{1}{2}\sum_{i=1}^N\sum_{j=1}^N\omega^i\omega^j\psi'(x^i-x^j) 	\\
&= -\frac{1}{2}\sum_{j=1}^N\sum_{i=1}^N\omega^j\omega^i\psi'(x^j-x^i) + \frac{1}{2}\sum_{i=1}^N\sum_{j=1}^N\omega^i\omega^j\psi'(x^j-x^i) = 0.
\end{align}
Here we recall that $\psi'$ is odd since $\psi'(x) = x\phi(x)$ and $\phi$ is assumed to be even.

\noindent\textit{Step 2:} It suffices to consider the convergence of the right-most particle $x^k$. The same computation will give the convergence of the left-most particle $x^l$. We compute
\begin{align}
&\quad \frac{\dx{}}{\dx{t}}\left(\frac{1}{2}\left|
x^k - \sum_{j=1}^N\omega^j x_0^j
\right|^2\right) = -\left(
x^k - \sum_{j=1}^N\omega^j x^j
\right)\sum_{m=1}^N\omega^m \psi'(x^k-x^m) 	\\
&= - \left(
\sum_{j=1}^N \omega^j(x^k - x^j)
\right)\sum_{m=1}^N \omega^m (x^k-x^m)\phi(x^k-x^m) 	\\
&\leq -\delta \frac{1}{2}\left(
\sum_{j=1}^N \omega^j(x^k - x^j)
\right)^2 = - \delta \frac{1}{2}\left|
x^k - \sum_{j=1}^N\omega^j x_0^j
\right|^2.
\end{align}
We explain the appearance of $\delta>0$ in the last line. We assumed that the initial particle cloud has diameter strictly less than $R$ which leads to the estimate, using the contracting property of Corollary~\ref{cor:contractparticle}
\[
\sup_{m=1,\dots,N}|x^k-x^m| \leq \sup_{m=1,\dots,N}|x_0^k-x_0^m| \leq \sup_{i,j=1,\dots,N}|x_0^i-x_0^j| = D < R.
\]
Hence, we have a lower bound of $\phi(x^k-x^m)$ for any $m=1,\dots,N$ by
\[
\inf_{m=1,\dots,N}\phi(x^k-x^m) \geq \inf_{m=1,\dots,N}\phi(x_0^k-x_0^m)\geq \inf_{x\in[-D,D]}\phi(x) =: \delta/2>0.
\]
We are guaranteed that $\delta>0$ by assumption that $\phi >0$ on $(-R,R)$ and the regularity of $\phi$. We conclude by Gr\"onwall's inequality
\[
\left|
x^k - \sum_{j=1}^N\omega^j x_0^j
\right|^2 \leq \left|
x_0^k - \sum_{j=1}^N\omega^j x_0^j
\right|^2e^{-\delta t}.
\]
\noindent \textit{Step 3:} We take the convergence result of Step 2 at the level of the characteristic system~\eqref{eq:particleevolution} and interpret this in the context of the original PDE~\eqref{eq:GPY}. We will use the well-known dual form of the 1-Wasserstein distance. Fix a Lipscthitz test function $\tau$ with $|\tau'|\leq 1$. By ~\cref{lem:particleevolution}, we test the solution $\rho_t = \sum_{i=1}^N\omega^i \delta_{x_t^i}$ against $\tau$
\begin{align}
&\quad \left|\langle \rho_t,\tau \rangle - \left\langle \delta_{\sum_{j=1}^N\omega^j x_0^j},\tau\right\rangle\right| = \left|
\sum_{i=1}^N\omega^i\left[
\tau(x^i) - \tau\left(
\sum_{j=1}^N\omega^j x_0^j
\right)
\right]
\right| 	\\
&\leq 
\sum_{i=1}^N\omega^i\left|
\tau(x^i) - \tau\left(
\sum_{j=1}^N\omega^j x_0^j
\right)
\right| \leq \sum_{i=1}^N\omega^i \left|
x^i - \sum_{j=1}^N\omega^j x_0^j
\right||\tau'(\xi^i)| 	\\
&\leq \sum_{i=1}^N \omega^i \left|
x_0^i - \sum_{j=1}^N\omega^j x_0^j
\right|e^{-(\delta/2) t} \lesssim e^{-(\delta/2) t}.
\end{align}
In the second line, we have used the Mean-Value inequality where $\xi^i$ is some number between $x^i$ and $\sum_{j=1}^N \omega^jx_0^j$ as well as the Lipschitz assumption of $\tau$. This shows the exponential convergence as $t\to\infty$ in 1-Wasserstein.
\end{proof}
By well-posedness theory with probability measures as initial data~\cite{CDFLS11} or a standard mean-field limit Dobrushin's argument in~\cite{Gol16}, the previous results can be extended to general probability measures.
\begin{proposition}
Consider $\phi\in C_c^\infty(\R)$ satisfying the same assumptions as in~\cref{ex:particle}. Take $\rho_0\in \mathcal{P}(\R)$ such that
\[
\diam(\supp(\rho_0)) \leq D < R,
\]
with mean $m = \int_{-\infty}^\infty x\dx{\rho}_0(x)$. Then, there is a unique weak solution $\rho\in C([0,\infty);\mathcal{P}(\R))$ to equation~\eqref{eq:GPY} with initial condition $\rho_0$. Furthermore, there is exponential convergence in 1-Wasserstein to consensus;
\[
\rho \to \delta_m, \quad t\to \infty.
\]
\end{proposition}
\subsection{Application to a case of non-unique curves of maximal slope}
\label{sec:nonuniq}
The purpose of this example is to emphasise the dependence of the convergence results and the $\omega$-limit sets to the initial data and the curves of maximal slope emanating from it.  Our definition of stationary states is independent of initial data and solutions, however in~\cref{thm:converge,thm:converge'}, it is possible that different curves of maximal slope from the same initial data might be converging to disjoint subsets of the set of stationary states. Consider the following interaction potential
\begin{align}
W(x)=
\begin{cases}
\textrm{sign}(\abs{x}-1)\abs*{\abs{x}-1}^{3/2} & \abs{x}>\frac{3}{4} \\
\dfrac{\abs{x}^2}{2}- \dfrac{5}{32} & \abs{x} \leq \frac{3}{4}
\end{cases}
\, ,
\end{align}
along with the associated aggregation equation given by
\begin{align}
\partial_t \rho = \partial_x \bra{\rho \partial_x W * \rho} \, .
\end{align}

Note that $W \in C^1(\R)$ (see~\eqref{eq:gradW}) and $W (x)\leq C\bra*{1+ \abs{x}^2}$ but it is not $\lambda$-convex for any $\lambda \in \R$. Thus, it is not a \emph{pointy potential} or a variation of it in the sense of~\cite{CDFLS11,CLM14}. We search for solutions $\rho \in AC([0,\infty); \cP_2(\R))$ that are $2$-curves of maximal slope of the energy $E: \cP_2(\R) \to (-\infty,+\infty]$ given by
\begin{align}
E(\mu)= \frac{1}{2} \iint_{\R \times \R}W(x-y) \mu(\dx{x}) \mu(\dx{y}) \, ,
\end{align}
with respect to the upper gradient $G : \cP_2(\R) \to [0, \infty]$ given by
\begin{align}
G(\mu)= \bra*{\int_\R \abs{\partial_x W * \mu}^2 \mu(\dx{x}) }^{1/2} \, ,
\end{align}
for some initial data $\rho(0)=\mu_0 \in \cP_2(\R)$.

Based on the forms of $E$ and $G$, it is easy to check that~\eqref{ass1'} and~\eqref{ass3'} are true. For simplicity, we will not address ~\eqref{ass2'} in full generality. In the setting of~\cref{prop:nonuniq}, the initial datum is symmetric on $\R$. Since the interaction potential, $W$, is also symmetric, any gradient flow solution is also symmetric. Therefore, we verify~\eqref{ass2'subset} with $\mathcal{Y}$ being  the set of all symmetric probability measures.
\begin{lemma}
\label{lem:nonuniqcpct}
The sublevel set $L_{\le C}(E) = \{ \mu \in \cP_2(\R) \cap \mathcal{Y} \, : \, E(\mu) \le C \}$ is relatively compact in $\cP(\R)$ where $\mathcal{Y}$ is the set of symmetric probability measures on $\R$.
\end{lemma}
\begin{proof}
The strategy for this proof is similar to~\cref{prop:exunminRd} borrowing methods from~\cite[Theorem 2.1]{CCV15} or~\cite{McC94}. We repeatedly make use of the fact that $W$ is an increasing function of $|x|$. We define first the reduced energy
\[
E_1(\mu) := \frac{1}{2}\int_{-\infty}^\infty \int_{|x-y|\ge 1}W(x-y) \mu(\dx{x})\mu(\dx{y}).
\]
If $E(\mu) \le C$ for $C\in \R$, then we can also upper bound $E_1(\mu)$ using $W\ge -\frac{5}{32}$
\begin{align}
E_1(\mu) &= E(\mu) - \frac{1}{2}\int_{-\infty}^\infty \int_{|x-y|< 1}W(x-y) \mu(\dx{x})\mu(\dx{y}) \le C + \frac{5}{64}\int_{-\infty}^\infty \int_{|x-y| < 1}\mu(\dx{x})\mu(\dx{y})  \\
&\le C + \frac{5}{64}.
\end{align}
For fixed $R\ge 1$ and $x\in \R$ such that $|x| \ge 1$, we have
\[
\{y \in \R \, : \, xy \le 0\} \subset \{y \in \R \, : \, |x-y| \ge 1\}.
\]
This provides a lower bound for $E_1(\mu)$ by
\[
E_1(\mu) \ge \int_{|x| > R}\int_{xy\le 0}W(x-y) \mu(\dx{x})\mu(\dx{y}).
\]
In fact, $xy \le 0$ implies the better estimate $|x-y| \ge |x|$ which minorises $W$
\begin{align}
E_1(\mu) &\ge \int_{|x|>R}\int_{xy\le 0} W(x) \mu(\dx{x})\mu(\dx{y})  \ge W(R)\int_{|x|> R}\int_{xy \le 0}\mu(\dx{x})\mu(\dx{y}) = \frac{W(R)}{2}\int_{|x|>R}\mu(\dx{x}).
\end{align}
The last equality uses the symmetry of $\mu$; $\{y \in \R \, : \, xy\le 0\}$ is one of $(-\infty, 0]$ or $[0,+\infty)$, both of which have $\mu$-mass $1/2$. The lower and upper bounds of $E_1(\mu)$ just computed give the estimate
\[
\sup_{\mu \in L_{\le C}(E)}\int_{|x|> R}\mu(\dx{x}) \le \frac{2(C+\frac{5}{64})}{W(R)} \overset{R\to \infty}{\to} 0.
\]
This establishes compactness by Prokhorov's theorem.
\end{proof}
~\cref{lem:nonuniqcpct} verifies the remaining assumption in our abstract result.
\begin{proposition}
Let $\mu \in AC([0,\infty);\cP_2(\R))$ be a 2-curve of maximal slope of the energy $E$ with respect to the weak upper gradient $G$, for initial data $\mu_0\in Z_E \cap \cY$. Then,
\[
\lim_{t\to \infty} d_{LP}(\mu(t), \mathcal{E}_w) = 0,
\]
where $\mathcal{E}_w \subset \cP(\R)$ is the set of weak stationary states of the 2-curves of maximal slope associated to $E$ and $G$.
\end{proposition}
\begin{proof}
This is an application of~\cref{thm:converge'} since~\cref{lem:nonuniqcpct} is precisely~\eqref{ass2'subset}. The other assumptions~\eqref{ass1'} and~\eqref{ass3'} are easily verified based on the forms of $E$ and $G$. 
\end{proof}
We emphasise here that our abstract convergence result is between a curve and a set; different curves of maximal slope could converge to disjoint subsets of the set of stationary states. We illustrate this in the proof of ~\cref{prop:nonuniq} below.
\begin{proposition}
\label{prop:nonuniq}
Let $\mu_0= \dfrac{1}{2}(\delta_{\frac{1}{2}}+\delta_{-\frac{1}{2}}) \in \cP_2(\R)$. Then, there exist two distinct curves $\rho,\eta \in AC([0,\infty); \cP_2(\R))$ with $\rho(0)=\eta(0)= \mu_0$ such that $\rho$ and $\eta$ are both curves of maximal slope of the energy $E$ with respect to the upper gradient $G$. Furthermore,
\begin{align}
\lim_{t \to \infty}d_{LP}(\rho(t),\mu_0)
=0 \qquad \textrm{and} \qquad \lim_{t \to \infty}d_{LP}(\eta(t),\delta_0)
=0 \,. 
\end{align}
\end{proposition}
\begin{proof}
We first compute the gradient of $W$ as follows
\begin{align}
\label{eq:gradW}
W'(x)=
\begin{cases}
\textrm{sign}(x)\dfrac{3}{2}\abs*{\abs{x}-1}^{1/2} & \abs{x}>\frac{3}{4} \\
 x & \abs{x}\leq \frac{3}{4}
\end{cases}
\, .
\end{align}
Note now that $\rho \equiv \mu_0$ is a curve of maximal slope of $E$ with respect to $G$. Indeed, we can check that
$G(\rho)=G(\mu_0)=0$. Thus, by~\cref{lemma:equil}, we have that $\rho$ is a (stationary) curve of maximal slope. To construct $\eta$ we first analyse the following ODEs
\begin{align}
\begin{cases}
\dfrac{\dx{x_1}}{\dx{t}}&= - \dfrac{1}{2} W'(x_1 -x_2) \\
&\\
\dfrac{\dx{x_2}}{\dx{t}}&= - \dfrac{1}{2}W'(x_2 -x_1)
\end{cases}
\, , \label{eq:ODE}
\end{align}
with $x_1(0)=-1/2,x_2(0)=1/2$. We define $x_3=x_2-x_1$ and $x_4=x_1$, which in turn satisfy the following system of ODEs
\begin{align}
\begin{cases}
\dfrac{\dx{x_3}}{\dx{t}}&= -W'(x_3) \\
&\\
\dfrac{\dx{x_4}}{\dx{t}}&=  \dfrac{1}{2}W'(x_3)
\end{cases}
\, ,
\end{align}
with $x_3(0)=1,x_4(0)=-1/2$. Thus,if we solve for $x_3$, $x_4$ (and by extension $x_1$ and $x_2$) can be obtained easily. Note that
\begin{align}
x_3(t)=
\begin{cases}
1- \frac{9}{16}t^2 & 0 \leq t \leq 2/3 \\
\frac{3}{4}e^{-(t-2/3)} & 2/3<t<\infty 
\end{cases}
\end{align}
is a solution to the ODE. Further calculations tell us that
\begin{align}
x_1(t)&=
\begin{cases}
-1/2+ \frac{9}{32}t^2 & 0 \leq t \leq 2/3 \\
-\frac{3}{8}e^{-(t-2/3)} & 2/3<t<\infty 
\end{cases} \\
x_2(t)&=
\begin{cases}
1/2- \frac{9}{32}t^2 & 0 \leq t \leq 2/3 \\
\frac{3}{8}e^{-(t-2/3)} & 2/3<t<\infty 
\end{cases}
\end{align}
are solutions of~\eqref{eq:ODE}. Define the curve
\begin{align}
\eta:= \frac{1}{2}\bra*{\delta_{x_1(t)} +\delta_{x_2(t)}} \, .
\end{align}
We will now argue that $\eta \in AC([0,\infty);\cP_2(\R))$ is a curve of maximal slope. We have
\begin{align}
\frac{\dx{}}{\dx{t}}E(\eta) &= \frac{1}{4}\frac{\dx{}}{\dx{t}}W(x_2(t)-x_1(t)) = \frac{1}{4}\frac{\dx{}}{\dx{t}}W(x_3(t))= -\frac{1}{4}\abs{W'}^2(x_3)  \, ,
\end{align}
where for the last equality we have simply used the fact that $x_3$ is the Euclidean gradient flow of $W$.
Additionally,
\begin{align}
\frac{1}{2}G^2(\eta)= \frac{1}{2}\bra*{\frac{1}{8}\abs{ W'}^2(x_2(t)-x_1(t)) + \frac{1}{8}\abs{ W'}^2(x_1(t)-x_2(t))} =\frac{1}{8}\abs{W'}^2(x_3(t)) \, .
\end{align}
Finally, we are left to compute the metric derivative of $\eta$. 
\begin{align}
d_2^2(\eta(t),\eta(s)) = \frac{1}{2}\bra*{\abs{x_2(t)-x_2(s)}^2 + \abs{x_1(t)-x_1(s)}^2 } \, ,
\end{align}
for $0 \leq s \leq t <\infty$ and $\abs{t-s} \ll 1$. The above equality follows by explicitly constructing the transport map and can be simplified using symmetry to
\begin{align}
d_2^2(\eta(t),\eta(s)) = \abs{x_2(t)-x_2(s)}^2 \, .
\end{align}
Dividing by $\abs{t-s}$ and passing to the limit as $t \to s$ we obtain
\begin{align}
\lim_{t \to s}\frac{d_2(\eta(t),\eta(s))}{\abs{t-s}}= \lim_{t \to s}\frac{\abs{x_2(t)-x_2(s)}}{\abs{t-s}} =\frac{1}{2}\abs{W'}(x_3(s)) \, ,
\end{align}
from which we obtain that
\begin{align}
d_2(\eta(t),\eta(s)) = \int_s^t \frac{1}{2}\abs{ W'}(x_3(u)) \dx{u} \, .
\end{align}
Thus,$\frac{1}{2}\abs{W'}(x_3(t))$ is necessarily the metric derivative of $\eta$ and
\begin{align}
\frac{1}{2}\abs{\eta'}^2(t)= \frac{1}{8}\abs{W'}^2(x_3(t)) \, .
\end{align}
Thus,it holds that
\begin{align}
\frac{\dx{}}{\dx{t}}E(\eta(t))=-\frac{1}{2}\abs{\eta'}^2(t) - \frac{1}{2}G^2(\eta(t)) \, ,
\end{align}
from which the result follows.
\end{proof}

\section*{Acknowledgements}
	JAC was supported the Advanced Grant Nonlocal-CPD (Nonlocal PDEs for Complex Particle Dynamics: 	Phase Transitions, Patterns and Synchronization) of the European Research Council Executive Agency (ERC) under the European Union's Horizon 2020 research and innovation programme (grant agreement No. 883363). JAC and RG were also supported by EPSRC grant number EP/P031587/1. RG and JW were supported by the President's PhD Scholarship Award from Imperial College.

\appendix

\section{Dynamical systems}
\label{sec:dynsys}
In this section, we discuss some preliminaries from the theory of the continuous-time metric dynamical systems. We follow the discussion in~\cite[Chapter 9]{CH98}. We make slight modifications to deal with the fact that trajectories of the dynamical systems we consider are not necessarily unique. Let $(Z,d)$ be a complete metric space. We start with the definition of a metric dynamical system:
\begin{defn}[Metric dynamical system]
\label{def:mds}
A \emph{metric dynamical system}  on $Z$ is a family of mappings $\set{S_t}_{t \geq 0}$  on $Z$ such that
\begin{enumerate}
\item  $S_t: Z \to 2^Z$ for all $t \geq 0$; \label{metdyn1}
\item $S_0= I$ where $I: Z \to 2^Z, z \mapsto \set{z}$; \label{metdyn2}
%\item $S_{t+s}^i= S_t^i \circ S_s^i$ for all $s,t \geq 0$  and $i \in I$;
\item For all $z_0 \in Z$ there exists a family of trajectories $ \bra{t \mapsto z_j(t)} \in C([0,\infty); Z), j \in J$ for some, possibly uncountable, index set $J$, such that $S_t z= \bigcup_{j \in J} \set{z_j(t)}$. \label{metdyn3}
\end{enumerate}
 A \emph{uniquely-defined metric dynamical system} is a continuous selection of the  family of mappings $\set{S_t}_{t \geq 0}$  on $Z$ such that
 \begin{enumerate}
\item  $S_t \in C(Z;Z)$ for all $t \geq 0$;
\item $S_0= I$ ;
\item $S_{t+s}= S_t \circ S_s$ for all $s,t \geq 0$;
\item For all $z_0 \in Z$ there exists a trajectory $ \bra{t \mapsto z(t)} \in C([0,\infty); Z)$  such that $S_t z_0=  z(t)$.
\end{enumerate}
\end{defn}
We now introduce the concept of the $\omega$-limit set associated to a point $z \in Z$.
\begin{defn}[$\omega$-limit set]
Let $\set{S_t}_{t \geq 0}$ be a metric dynamical system and $z_0 \in Z$. Then,  the set
\begin{align}
\omega^j(z_0):= \set*{z^* \in Z: \exists t_n \to \infty, \lim_{n \to \infty} d(z_j\bra{t_n} ,z^*)=0}
\end{align}
is called the \emph{$\omega$-limit set of $z_0$ subordinate to $j \in J$}. If the dynamical system $\set{S_t}_{t \geq 0}$ is uniquely-defined, then we call
\begin{align}
\omega(z_0):= \set*{z^* \in Z: \exists t_n \to \infty, \lim_{n \to \infty} d(z\bra{t_n} ,z^*)=0} \, ,
\end{align}
the \emph{$\omega$-limit set of $z_0$}.
\label{def:wls}
\end{defn}
We then have the following result:
\begin{theorem}\label{thm:ch}
Let $\set{S_t}_{t \geq 0}$ be a metric dynamical system and $z_0 \in Z$. Assume that $\bigcup_{t \geq 0} \set{z_j(t)}$ is relatively compact in $Z$ for some $j \in J$. Then,
\begin{tenumerate}
\item $\omega^j(z_0)$ is non-empty, compact, and connected. \label{thm:ch!a}
\item $\lim_{t \to \infty} d(z_j(t),\omega^j(z_0))=0 $, where 
$$
d(z_j(t),\omega^j(z_0)):=\inf_{z^* \in \omega^j(z_0)}d(z_j(t),z^*)=\min_{z^* \in \omega^j(z_0)}d(z_j(t),z^*) \, . \label{thm:ch!b}
$$
\end{tenumerate}
Assume $\set{S_t}_{t \geq 0}$ is uniquely-defined and $\bigcup_{t \geq 0} \set{z(t)}=\bigcup_{t \geq 0} \set{S_t z_0}$ is relatively compact. Then, the same results hold for $\omega(z_0)$.
\end{theorem}
\begin{proof}
We note that the proof of~\Dref{thm:ch!a} follows from the fact that we can find subsequence $(z_i\bra{t_n})_{n \in \N}$ and $z^* \in Z$ such that 
$$
\lim_{n \to \infty} d(z_j\bra{t_n},z^*)=0 \ .
$$
thus, $z^* \in \omega^j(z_0)$ and $\omega^j(z_0) \neq \emptyset$. Furthermore, by definition, we have that
\begin{align}
\omega^j(z_0)= \bigcap_{s>0}\overline{\bigcup_{t \geq s} \set{z_j\bra{t}}} \, .
\end{align}
$\omega^j(z_0)$ is a decreasing intersection of compact and connected sets and thus is compact and connected itself.

For the proof of~\Dref{thm:ch!b}, we assume by contradiction that we can find a sequence $t_n \to \infty$ such that 
$$
\liminf_{n \to \infty} d(z_j\bra{t_n},\omega^j(z_0)) \geq \delta \, .
$$
By compactness, there exists a subsequence $t_{n_k} \to \infty$ and some $z^*$ such that
$$
\lim_{k \to \infty} d(z_j\bra{t_{n_k}},z^*) =0 \, .
$$
Clearly, $z^* \in \omega^j(z_0)$ which is absurd. Thus,~\ref{thm:ch!b} follows.
\end{proof}
\section{Lower semicontinuity of weak upper gradients}
\label{sec:lsc}
In this section we include a short justification on a general proof strategy for deducing lower semicontinuity of the weak upper gradients that share certain structural features. The abstract setting is as follows: $X$ is $\R^d$ or $\T^d$ space and $G: \cP(X) \to (-\infty,+\infty]$ is a proper function with domain
\begin{align}
D(G)=\set{\rho \in \cP(X): G(\rho) < + \infty} \, .
\end{align}
We assume that $G$ has the following form:
\begin{align}
G(\rho):=\bra*{\int_{X} \abs{v_{\rho}(x)}^p \dx{\rho}(x)}^{\frac{1}{p}} \, , \label{eq:LSC} \tag{LSC}
\end{align}
for some $1<p < \infty$ and some Borel measurable $v_\rho: X \to \R^d$ dependent on $\rho$. We then have the following result:
\begin{theorem}
\label{thm:LSC}
The function $G: \cP(X) \to (-\infty,+\infty]$ is l.s.c.
\end{theorem}
\begin{proof}
We need to check that $G$ is l.s.c with respect to weak convergence of probability measures, i.e. given $ \rho_n \to \rho $ in $\cP(X)$ it holds that
\begin{align}
G(\rho) \leq \liminf_{n \to \infty} G(\rho_n) \, . \label{eq:Gliminf}
\end{align}
Assume first that $\rho_n$ has no subsequence $\rho_{n_k}$ such that $\sup\limits_{n_k \in \N} G(\rho_{n_k}) < \infty$. Then, we have that
\begin{align}
\liminf_{n \to \infty}G(\rho_n) = +\infty \, ,
\end{align} 
and thus~\eqref{eq:Gliminf} follows trivially. Indeed, assume this is not the case, i.e.
 $$
 \liminf_{n \to \infty}G(\rho_n)< C <\infty \,.
 $$
  Then, for some $\eps>0$ we can find an $N_0 <\infty$, such that for all $N \geq N_0$ it holds that
\begin{align}
\inf_{N \geq N_0} G(\rho_n) \leq C+ \eps/2 \, .
\end{align}
Thus, there exists some $n_1 \geq N$ such that $G(\rho_{n_1}) \leq C+ \eps$. Now setting $N_1= n_1+1$, we have that
\begin{align}
\inf_{N \geq N_1} G(\rho_n) \leq C+ \eps/2 \, .
\end{align}
Thus, there exists some $n_2 >n_1$ such that $G(\rho_{n_2}) \leq C+ \eps$. Proceeding like this, we can construct a subsequence $\rho_{n_k}$ such that $G(\rho_{n_k}) \leq C+ \eps $, which provides us with a contradiction. Thus, it is sufficient to check lower semicontinuity of $G$ along sequences with at least one bounded subsequence. In fact, we have that
\begin{align}
\liminf_{n \to \infty} G(\rho_n) = \inf \set*{\liminf_{n \to \infty} G(\rho_{n_k}): \sup_{n_k \in \N} G(\rho_{n_k}) < \infty }
\label{eq:liminfcharacterisation}
\end{align} 
Indeed, we have the bound
\begin{align}
\liminf_{n \to \infty} G(\rho_n) = \lim_{N \infty} \inf_{n \geq N} G(\rho_n) \leq \lim_{N \to \infty}\inf_{n_k \geq N} G(\rho_{n_k}) \, ,
\end{align}
for an arbitrary bounded subsequence $\rho_{n_k}$. Taking the infimum the we have the bound in one direction. For the other direction we note that, as before, for all $\eps>0$ we can find a bounded subsequence $\rho_{n_k}$ such that
\begin{align}
\liminf_{n_k \to \infty} G(\rho_{n_k}) \leq  \liminf_{n \to \infty} G(\rho_{n}) + \eps \, . 
\end{align}   
Thus,~\eqref{eq:liminfcharacterisation} follows. For every $G$-bounded sequence $(\rho_n)_{n \in \N}$ weakly convergent to some $\rho \in \cP(X)$, we  set $v_n(x)=v_{\rho_n}(x)$  and note that
\begin{align}
\sup_{n \in \N} \norm{v_n}_{\Leb^p(X,\rho_n)} < \infty \, .
\end{align}
Applying~\cite[Theorem 5.4.4]{AGS08}, we have that
\begin{align}
G(\rho) \leq \liminf_{n_k \to \infty}G(\rho_{n}) \, .
\end{align}
Taking the infimum over all bounded subsequences and using~\eqref{eq:liminfcharacterisation}, the result of the theorem follows.
\end{proof}

\bibliographystyle{myalpha}
\bibliography{refs}

\end{document}